\documentclass[10pt]{amsart}
\usepackage{amsmath, amscd, amsthm} 
\usepackage{amssymb, amsfonts} 
\usepackage{enumerate}
\usepackage[all]{xy} %To use xy-pic for diagrams.
\usepackage[margin=1.5in]{geometry}%changing the margins so the table
\usepackage{palatino}
\subjclass[2010]{Primary: 14F42 Secondary: 55P42, 19D45, 18E30}
\keywords{Tensor triangular geometry, stable motivic homotopy theory}
\usepackage{mathrsfs}
\usepackage{booktabs}  
\usepackage{aliascnt}
\usepackage{mathrsfs} 
\usepackage[svgnames]{xcolor} 
\usepackage{hyperref}
\usepackage{graphicx}
 \definecolor{dark-red}{rgb}{0.4,0.15,0.15}
\hypersetup{
    colorlinks, linkcolor=dark-red,
    citecolor=DarkBlue, urlcolor=MediumBlue
}
\newcommand{\Q}{\mathbb{Q}} 
\newcommand{\QQ}{\Q}
\newcommand{\CC}{\mathbb{C}} 
\renewcommand{\AA}{\mathbb{A}}
\newcommand{\Cc}{\mathscr{C}}
\newcommand{\Dd}{\mathscr{D}}

\newcommand{\Z}{\mathbb{Z}}
\newcommand{\ZZ}{\mathbb{Z}}
\newcommand{\R}{\mathbb{R}}
\newcommand{\RR}{\R}

\newcommand{\FF}{\mathbb{F}}
\newcommand{\KK}{\mathscr{K}}

\newcommand{\iso}{\cong}
\newcommand{\wkeq}{\simeq}

\newcommand{\Sm}{\mathrm{Sm}}

\newcommand{\SH}{\mathrm{SH}}
\newcommand{\HH}{\mathbb{H}}

\newcommand{\Pic}{\operatorname{Pic}}

\newcommand{\supp}{\operatorname{supp}}

\newcommand{\id}{\mathrm{id}}

\newcommand{\ul}[1]{\underline{\smash{#1}}}

\renewcommand{\setminus}{\smallsetminus}

\DeclareMathOperator{\Sing}{\mathcal{S}}

\newcommand{\fin}{\mathrm{fin}}
\newcommand{\cell}{\mathrm{cell}}

\usepackage[T1]{fontenc}

\numberwithin{equation}{section} %Fiddles with numbering system of the following.

\theoremstyle{plain}

\newaliascnt{theorem}{equation}  
\newtheorem{theorem}[theorem]{Theorem}  
\aliascntresetthe{theorem}

 \theoremstyle{definition}

\newaliascnt{prop}{equation}  
\newtheorem{prop}[prop]{Proposition}
\aliascntresetthe{prop}

\newaliascnt{lemma}{equation}  
\newtheorem{lemma}[lemma]{Lemma}
\aliascntresetthe{lemma}

\newaliascnt{corollary}{equation}  
\newtheorem{corollary}[corollary]{Corollary}
\aliascntresetthe{corollary}

\newaliascnt{claim}{equation}  

\aliascntresetthe{claim}

\newaliascnt{conjecture}{equation}  

\aliascntresetthe{conjecture}

\newaliascnt{question}{equation}  
\newtheorem{question}[question]{Question}
\aliascntresetthe{question}

\newaliascnt{defn}{equation}  
\newtheorem{defn}[defn]{Definition}
\aliascntresetthe{defn}

\newaliascnt{example}{equation}  

\aliascntresetthe{example}

\theoremstyle{remark}

\newaliascnt{remark}{equation}  
\newtheorem{remark}[remark]{Remark}
\aliascntresetthe{remark}

\newcommand{\aref}[1]{\autoref{#1}}

\newcommand{\Spc}{\operatorname{Spc}}
\renewcommand{\AA}{\mathbb{A}}
\newcommand{\SHA}{\mathrm{SH}^{\AA^1}\!}
\newcommand{\SHAt}[1]{\mathrm{SH}^{\AA^1}_{#1}}
\newcommand{\EM}{\mathrm{EM}}
\newcommand{\kmw}[2]{K^{MW}_{#1}(#2)}
\newcommand{\PP}{\mathbb{P}}

\newcommand{\Spec}{\operatorname{Spec}}
\newcommand{\pp}{\mathfrak{p}}
\newcommand{\GrFrac}{\operatorname{GrFrac}}

\newcommand{\Pp}{\mathscr{P}}
\newcommand{\Cell}{\operatorname{Cell}}
\newcommand{\cone}{\operatorname{cone}}
\newcommand{\RE}{\operatorname{Re}}

\newcommand{\chr}{\operatorname{char}}

\begin{document}
\title{Primes and fields in stable motivic homotopy theory}
\author{J. Heller}
\address{University of Illinois Urbana-Champaign}
\email{jbheller@illinois.edu}
\author{K. Ormsby}
\address{Reed College}
\email{ormsbyk@reed.edu}
% \address{}
% \thanks{}

\begin{abstract}
Let $F$ be a field of characteristic different than $2$. We establish surjectivity of Balmer's comparison map 
\[
  \rho^\bullet:\Spc(\SHA(F)^c)\to \Spec^h(K^{MW}_*(F))
\]
from the tensor triangular spectrum of the homotopy category of compact motivic spectra to the homogeneous Zariski spectrum of Milnor-Witt $K$-theory.
We also comment on the tensor triangular geometry of compact cellular motivic spectra, producing in particular novel field spectra in this category.  We conclude with a list of questions about the structure of the tensor triangular spectrum of the stable motivic homotopy category.
\end{abstract}
 
\maketitle
% \tableofcontents

\section{Introduction}\label{sec:intro}
  
The chromatic approach to stable homotopy theory is a powerful organizational and computational tool. Underlying this approach 
is a good understanding of the tensor triangular geometry of the stable homotopy category, as provided by the celebrated nilpotence and periodicity theorems of Hopkins-Smith \cite{hs:nil2}. Recent computational work \cite{Andrews:nonnilpotent} suggests a rich and vast chromatic picture in the motivic setting. However, a good (or any) understanding of the tensor triangular geometry of the stable motivic homotopy category is still lacking.

To set the stage for our results, recall that given a (small) tensor triangulated category $\KK$, the set underlying Balmer's spectrum $\Spc(\KK)$ is the set of tensor triangular prime ideals \cite{Balmer:ttspectrum}. This simple invariant is surprisingly powerful;  for example, under mild hypotheses it completely determines the thick tensor ideals of $\KK$ 
\cite[Theorem 4.10]{Balmer:ttspectrum}. Computing the spectrum of $\KK$ is no easy task, but Balmer's comparison maps \cite[Theorem 5.3, Corollary 5.6]{balmer:sss} make it possible to organize $\Spc(\KK)$ fiberwise over the classical spectra of (homogenous) prime ideals of the (graded) endomorphism ring of the unit.

This plays out as follows in the motivic setting over a field $F$ whose characteristic is not $2$. The tensor triangulated categories of interest here are the homotopy category of compact motivic $\PP^1$-spectra  $\SHA(F)^{c}$ and the homotopy category of compact cellular motivic $\PP^1$-spectra  $\Cc(F)$. We write $\KK(F)$ for either one of these two categories. Of course the graded endomorphism ring (relative to $\AA^1\smallsetminus 0$) is the same in either case, and by a theorem of Morel  \cite{morel:pi0}, this ring is the Milnor-Witt $K$-theory ring $\kmw*F$ of $F$.  In degree zero, $K^{MW}_{0}(F) = GW(F)$, the Grothendieck-Witt ring of quadratic forms.  Balmer's comparison maps thus take the form 
$$
\xymatrix{
 \Spc(\KK(F)) \ar[r]^-{\rho^\bullet}\ar[rd]_{\rho} &  \Spec^h(\kmw*F) \ar[d] \\
  & \Spec(GW(F)).
} 
$$

Moreover, \cite[Theorem 7.13]{balmer:sss} applies to show that $\rho$ is surjective. Complete knowledge of $\Spec(GW(F))$ can be extracted from classical work of Lorenz-Leicht \cite{lorenzleicht}, see also \cite[Remark 10.2]{balmer:sss}. Thus $\rho$ provides a very explicit lower bound for $\Spc(\KK(F))$.

Recent work of Thornton \cite{thornton} shows that $\Spec^h(\kmw*F)$ has precisely $|X_F|+1$ more elements than $\Spec(GW(F))$, where $X_F$ denotes the set of orderings on the field $F$. Thus $\rho^{\bullet}$ could potentially detect up to $|X_F|+1$ more primes of $\KK(F)$ than $\rho$ can. Our main thereom is that $\rho^{\bullet}$ is surjective and thus $\rho^{\bullet}$ does yield a full $(|X_F|+1)$-fold improvement on the previously known lower bounds for $\Spc(\SHA(F)^c)$. This meets the challenge posed in \cite[Remark 10.5]{balmer:sss}.

\begin{theorem}\label{thm:main}
Suppose $F$ is a field whose characteristic is not $2$.  Then both
$$
\rho^\bullet :\Spc(\SHA(F)^{c}) \longrightarrow \Spec^h(\kmw*F)
$$
and
$$
  \rho^\bullet :\Spc(\Cc(F)) \longrightarrow \Spec^h(\kmw*F)
 $$
are surjective.
\end{theorem}

\begin{remark}
This theorem tells us that we have a commutative diagram
\[\xymatrix{
  \Spc(\SHA(F)^c)\ar[r]\ar[d]_{\rho^\bullet} &\Spc(\Cc(F))\ar[d]^{\rho^\bullet}\\
  \Spec^h(\kmw*F)\ar@{=}[r] &\Spec^h(\kmw*F)
}\]
in which both vertical arrows are surjective.  Balmer has informed us of a general criterion \cite[Theorem 1.3]{balmer:surj} for when a tensor triangulated functor $f:\KK\to \mathscr{L}$ between essentially small tensor triangulated categories induces a surjection $\Spc(f):\Spc(\mathscr{L})\to\Spc(\KK)$; namely, if $f$ detects $\otimes$-nilpotence and $\KK$ is rigid, then $\Spc(f)$ is surjective.  This condition is clearly met by $f:\Cc(F)\subset \SHA(F)^{c}$, so the top horizontal arrow in the above diagram is surjective as well.
\end{remark}

Our argument for \aref{thm:main} is indirect, and does not produce any explicit prime ideals in the stable motivic homotopy categories.  After establishing our main theorem we address the question of producing explicit primes  for $\SHA(F)^c$ and $\Cc(F)$ by analyzing realization functors (\aref{sec:betti}) and field spectra for $\Cc(F)$ (\aref{sec:explicit}).  Ultimately, when $F$ is isomorphic to a subfield of $\CC$ we find infinite towers of tensor primes above all but one positive characteristic prime in $\kmw*F$; see \aref{fig:sketch}.

Both Joachimi \cite[Ch.~7]{joa} and Kelly \cite{kelly} have studied the tensor triangular geometry of $\Spc(\SHA(F)^c)$.  Joachimi's study is primarily concerned with Betti realization functors, and should be compared with our results in \aref{sec:betti}.  Kelly proves that $\rho^\bullet$ is surjective when $F$ is a finite field; we adapt his argument to prove \aref{prop:nonreal}.

\subsection*{Outline}
We recall the basic background of tensor triangular geometry, stable motivic homotopy theory, and Milnor-Witt $K$-theory in \aref{sec:prelim}. In particular we recall there Thornton's theorem on the structure of $\Spec^h(\kmw*F)$. In \aref{sec:surj} we establish our main result, \aref{thm:main}, on the surjectivity of $\rho^{\bullet}$.  In \aref{sec:explicit}, we find explicit field spectra for $\Cc(F)$ whose acyclics are tensor triangular primes living over particular elements of $\Spec^h(\kmw*F)$, see \aref{thm:explicit}.  In \aref{sec:betti}, we completely determine the effect of $\rho^\bullet$ on the injective image of $\Spc(\SH^{\fin})$ and $\Spc(\SH(C_2)^c)$ under the maps induced by Betti realization, see \aref{prop:topmot} and \aref{prop:eqmot}.  In \aref{sec:fields}, we devote our attention to field spectra for $\SHA(F)^c$ and elucidate the manner in which singular and realization functors interact with tensor triangular primes.  We conclude with a list of questions in \aref{sec:q}, amounting to an optimist's sketch of how to determine the complete structure of $\Spc(\SHA(F)^c)$. (\emph{Caveat emptor}: some of these questions are likely quite difficult.) 

\subsection*{Acknowledgements} It is our pleasure to thank Shane Kelly and Paul Balmer for helpful discussions and for sharing with us drafts of \cite{kelly} and \cite{balmer:surj}, respectively.  We also thank the anonymous referee for helpful comments and improvements to the exposition.  Ormsby was partially supported by NSF award DMS-1406327.

\subsection*{Notation}
We use the following notation throughout. 
\begin{itemize}
\item $F$: a field, always of characteristic $\ne 2$.
\item $\Sm/F$: the category of  smooth finite type schemes over $F$.
\item $\SHA(F)$, $\SHA(F)^c$, $\Cc(F)$ are respectively the stable motivic homotopy category over $\Spec F$, the full subcategory of compact motivic spectra, and the full subcategory of compact cellular motivic spectra.  
\item $\KK(F)$ refers to either of $\SHA(F)^c$ or $\Cc(F)$.
\item $K^{MW}_*(F)$: the Milnor-Witt $K$-theory ring of $F$.  
\item $S^0$ is the motivic sphere spectrum over $\Spec F$.
\item $S^m$, $S^{m\alpha}$ are respectively the ``simplicial''
and the ``geometric'' $m$-sphere, $m\in \Z$;  $S^m$ is  the $m$-fold suspension of $S^0$, and $S^{m\alpha} = (\AA^1\smallsetminus 0)^{\wedge m}$.
\item $[X,Y]_{\mathscr{D}}$: the set of maps between objects $X$, $Y$ of a category $\mathscr{D}$.
\item $\ul\pi_m(X)_n$ for $X\in \SHA(F)$, $m$, $n\in \ZZ$ is  the Nisnevich sheaf on $\Sm/F$ associated with the assignment
\[   
  U\longmapsto [S^m\wedge U_+,S^{n\alpha}\wedge X]_{\SHA(F)}.
\]
Note that other sources may write $\ul\pi_{m-n\alpha}(X)$ or $\ul\pi_{m-n,-n}(X)$ for this sheaf.  We have chosen our notation to match Morel's in \cite{morel:t,morel:pi0}.
\item $\pi_m(X)_n$ for $X\in \SHA(F)$, $m$, $n\in \ZZ$: sections of $\ul\pi_m(X)_n$ on $\Spec F$.  
\item $\Spec(R)$: the Zariski spectrum of prime ideals in a commutative ring $R$.
\item $\Spec^h(R_*)$: the Zariski spectrum of homogeneous prime ideals in an $\varepsilon$-commutative $\ZZ$-graded ring $R_*$ (which means that $\varepsilon\in R_0$, $\varepsilon^2=1$, and $ab = \varepsilon^{mn}ba$ for $a\in R_m$, $b\in R_n$).
\item $\Spc(\Cc)$ for $\Cc$ a tensor triangulated category: Balmer's tensor triangular spectrum of $\Cc$.
\end{itemize}

\section{Preliminaries}\label{sec:prelim}
In this section we recall background notions and material from tensor triangular geometry and stable motivic homotopy theory used in the rest of the paper.
\subsection{Tensor triangular geometry}\label{sec:ttgeo}
We begin by recalling a few definitions and facts about the spectrum of a tensor triangulated category. We refer the reader to \cite{balmer:sss,balmer:icm} for full details.

A \emph{tensor triangulated category} $(\KK, \otimes, 1)$ is a triangulated category $\KK$ together with a symmetric monoidal product $\otimes$ which is bi-exact and a unit $1\in \KK$. For $\KK$ essentially small, Balmer associates its \emph{tensor triangular spectrum} $\Spc(\KK)$, which is a topological space defined as follows.

A \emph{tensor triangular prime ideal} (or simply a prime ideal) of $\KK$ is a proper thick subcategory $\Pp\subsetneq \KK$ which is
\begin{itemize}
\item a \emph{tensor ideal}: if $a\in \KK$ and $b\in \Pp$, then $a\otimes b\in \Pp$, and is
\item \emph{prime}: if for some $a,b\in \KK$ we have $a\otimes b\in \Pp$, then $a$ or $b$ is in $\Pp$.
\end{itemize}

The prime spectrum of $\KK$ is the set of primes, 
$$
\Spc(\KK):=\{\Pp\subsetneq \KK \mid \textrm{ $\Pp$ is prime}\}.
$$ 
Given an object $a\in \KK$, its support is 
$$
\supp(a):= \{\Pp \in \Spc(\KK) \mid a\not\in \Pp\}.
$$
Write $U(a) = \Spc(\KK)\smallsetminus \supp(a)$. 
Equip $\Spc(\KK)$ with the topology defined by letting the sets $U(a)$ for $a\in \KK$ be a basis for the  open sets. In particular, the support $\supp(a)$ of an object is always closed.

The prime spectrum is contravariantly functorial in $\KK$. That is, given a tensor exact functor $\phi:\KK \to \mathcal{L}$ there is an induced continuous map 
$\Spc(\phi):\Spc(\mathcal{L})\to \Spc(\KK)$ defined by $\Pp\mapsto \phi^{-1}(\Pp)$.

Of central interest  in this paper are the continuous maps \cite[Theorem 5.3, Corollary 5.6]{balmer:sss} constructed by Balmer which relate the spectrum of a tensor triangulated category to the spectrum of graded and ungraded endomorphism rings. 
Let $u\in \Pic(\KK)$ be a $\otimes$-invertible object and  let
$R^{*}_{\KK}$ be the graded endomorphism ring of the unit: $R^{i}_{\KK} = [1, u^{\otimes i}]_{\KK}$ for $i\in \ZZ$.  We write $R_{\KK}$ = $R^{0}_{\KK}$ for ungraded ring of ordinary endomorphisms of the unit.  Write  
$ \Spec^h(R^{*}_{\KK})$ for the Zariski spectrum of homogeneous prime ideals in this graded ring. 
We have continuous maps
$$ 
\xymatrix{
\Spc(\KK) \ar[r]^{\rho^{\bullet}}\ar[rd]_{\rho} & \Spec^h(R^{*}_{\KK}) \ar[d]^{(-)^{0}} \\
& \Spec(R_{\KK}),
}
$$
where $\rho^{\bullet}$ is defined by 
$$
\rho^{\bullet}(\Pp):= \{f\in R^{*}_{\KK} \mid \Pp\in \supp(\cone(f)) \}
$$
and  $(-)^{0}$ is defined by sending a homogenous prime $\mathfrak{p}$ to the prime ideal $\mathfrak{p}^{0} = \mathfrak{p}\cap R_{\KK}^{0}$ of homogeneous degree zero elements. Of course, the map $\rho^\bullet$ depends on the choice of invertible element $u$; in general, we will omit $u$ from our notation, but will write $\rho^\bullet_u$ when specification is necessary.

\subsection{Stable motivic homotopy theory and Milnor-Witt {$K$}-theory}

We fix throughout a base field $F$ which will always be assumed to be of characteristic different than two.
The tensor triangulated categories of interest in this paper stem from the homotopy category $\SHA(F)$ of motivic spectra over  $F$, developed in \cite{v:icm, jardine:motsymm, morel:t}. This category itself is too large for tensor triangular geometry and, as is standard in this situation, we focus on the homotopy category $\SHA(F)^c$ of compact motivic spectra.\footnote{Recall that an object $X$ in a triangulated category $\KK$ is \emph{compact} provided that  
$[X, \oplus_{I} Y_i]_{\KK} \cong \oplus_{I}[X, Y_i]_{\KK}$, where $I$ is any set.} 
In this case, the symmetric monoidal product is the smash product, $\wedge$, and the unit object is the sphere spectrum $S^0$.

There are many invertible objects in $\SHA(F)^c$, two of which are singled out. The first is the ordinary circle $S^1$ and the second is the geometric (or Tate) circle, 
$S^{\alpha} := \AA^1\smallsetminus 0$. 
We write $S^{m+n\alpha} = (S^1)^{\wedge m}\wedge (S^{\alpha})^{\wedge n}$.
The homotopy category of \emph{compact cellular spectra} 
$\Cc(F)\subseteq \SHA(F)^c$ is defined to be the 
thick subcategory  generated by $\{S^{m+n\alpha} \mid n,m\in \Z\}$. We write $\KK(F)$ for either of these categories.

Milnor-Witt $K$-theory $\kmw*F$ plays a distinguished role in stable motivic homotopy theory because it is isomorphic to a graded ring of endomorphisms of the unit object in $\KK(F)$ (at least when  $\chr F\ne 2$). Specifically, Morel \cite{morel:pi0} proves that 
\[
  \kmw*F\cong \pi_0(S^0)_* \cong [S^0,S^{*\alpha}]_{\KK(F)}.
\]
Milnor-Witt $K$-theory is defined to be a quotient of the free associative $\ZZ$-graded ring on the set of symbols $[F^\times] := \{[u]\mid u\in F^\times\}$ in degree $1$ and $\eta$ in degree $-1$; the quotient is by the homogeneous ideal enforcing the following relations:
\begin{itemize}
\item $[uv] = [u]+[v]+\eta[u][v]$ (twisted logarithm),
\item $[u][v] = 0$ for $u+v=1$ (Steinberg),
\item $[u]\eta=\eta[u]$ (commutativity), and
\item $(2+[-1]\eta)\eta = 0$ (Witt).
\end{itemize}
Milnor-Witt $K$-theory is $\varepsilon$-commutative for $\varepsilon = -(1+[-1]\eta)$.

Moreover, it is related as follows to several important classical invariants. 
\begin{itemize}
 \item $K^{MW}_{0}(F) \cong GW(F)$, the Grothendieck-Witt ring of quadratic forms.
 \item $\eta^{-1}\kmw*F\cong W(F)[\eta, \eta^{-1}]$, the ring of Laurent polynomials over the Witt ring.
 \item $\kmw*F/\eta = K^M_*(F)$, the Milnor $K$-theory of $F$.
\end{itemize}

As observed in \cite[Corollary 10.1]{balmer:sss}, the map $\rho:\Spc(\KK(F)) \to \Spec(GW(F))$ is surjective. (This follows from a general criterion \cite[Theorem 7.13]{balmer:sss} for surjectivity of $\rho$.)  
Our results concern the map 
$\rho^{\bullet}:\Spc(\KK(F))\to \Spec^h(\kmw*F)$ and rely
crucially on the following theorem of Thornton.  Recall that $X_F$ is the set of orderings on $F$ equipped with the Harrison topology.\footnote{A subbasis for the Harrison topology is given by the sets $H(a) = \{\alpha\in X_F\mid a >_\alpha 0\}$ for $a\in F^\times$.  The Harrison space $X_F$ is a Stone space; in particular, $X_F$ is discrete if and only if $|X_F|<\infty$.}

\begin{theorem}[{\cite[Theorem 3.12]{thornton}}]\label{thm:thornton}
Each $\pp\in \Spec^h(\kmw*F)$ is of exactly one of the following forms where $p$ ranges over rational primes, $\alpha$ ranges over orderings of $F$, and $P_\alpha\subseteq F^\times$ denotes the positive cone of $\alpha$:
\begin{enumerate}[(1)]
\item $([F^\times],\eta)$,
\item $([F^\times],\eta,p)$,
\item $([F^\times],2)$,
\item $([P_\alpha],h)$,
\item $([P_\alpha],\eta,2)$, and
\item $([P_\alpha],h,p)$ for $p\ne 2$.
\end{enumerate}
Additionally, the topology on $\Spec^h(\kmw*F)$ is specified by the following facts:
\begin{enumerate}[(a)]
\item The subspace of minimal prime ideals is homeomorphic to $X_F^* = X_F\amalg \{\infty\}$ where $X_F$ has the Harrison topology and $\infty$ is an isolated point.
\item The Hasse diagram (\emph{i.e.}, inclusion poset) of $\Spec^h(\kmw*F)$ is shown in \aref{fig:kmw} and the closure operator takes a subset $S\subseteq \Spec^h(\kmw*F)$ to everything in or above $S$ in the Hasse diagram.
\end{enumerate}
\end{theorem}

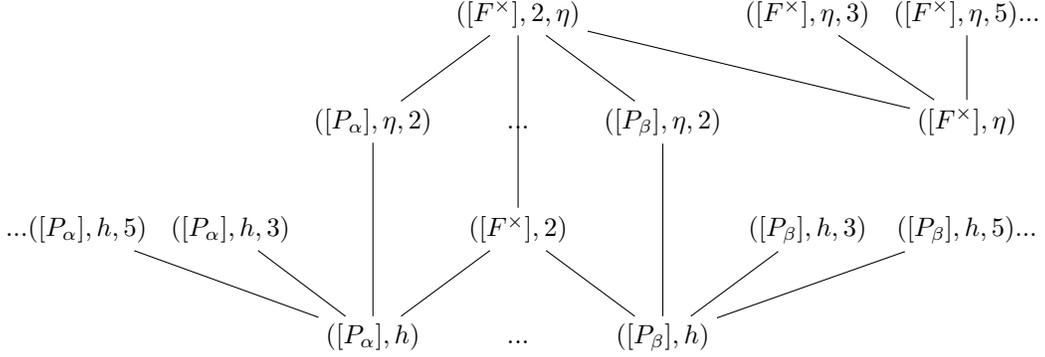
\begin{figure}
\[
\xymatrixcolsep{.25pc}\xymatrix{
 & & & ([F^\times],2,\eta)  \ar@{-}[ld] \ar@{-}[dd] \ar@{-}[rd] \ar@{-}[rrrd] & &([F^\times],\eta,3) \ar@{-}[rd] &([F^\times],\eta,5)... \ar@{-}[d] \\
 & &([P_\alpha],\eta,2)\ar@{-}[dd] & ... &([P_\beta],\eta,2)\ar@{-}[dd]& &([F^\times],\eta)\\
...([P_\alpha],h,5)\ar@{-}[rrd] &([P_\alpha],h,3)\ar@{-}[rd] & &([F^\times],2) \ar@{-}[ld] \ar@{-}[rd] & &([P_\beta],h,3)\ar@{-}[ld] &([P_\beta],h,5)...\ar@{-}[lld]\\
 & & ([P_\alpha],h)& ... &([P_\beta],h)}
\]
\caption{The Hasse diagram of $\Spec^h(\kmw*F)$.  Inclusion of tensor triangular primes and closure both go upwards in the diagram.  In the picture, we suppose that $\alpha$, $\beta\in X_F$, and the central ellipses represent the rest of $X_F$.}\label{fig:kmw}
\end{figure}

\section{Surjectivity}\label{sec:surj}
Let $\KK(F)$ denote either
$\SHA(F)^c$ or $\Cc(F)$.
  In this section we prove our main theorem \aref{thm:main}: the continuous map  $\rho^\bullet:\Spc(\KK(F))\to \Spec^h(\kmw*F)$ hits all the primes enumerated in \aref{thm:thornton}.  We warm up by showing that this is true when $F$ is nonreal, \emph{i.e.}~when $-1$ is a sum of squares in $F$.

\begin{prop}\label{prop:nonreal}
Suppose $F$ is a nonreal field whose characteristic is not $2$.  Then
\[
  \rho^\bullet:\Spc(\KK(F))\longrightarrow \Spec^h(\kmw*F)
\]
is surjective.
\end{prop}
\begin{proof}
Our proof is an adaptation of Kelly's argument for the surjectivity of $\rho^\bullet$ in the case when $F$ is a finite field \cite{kelly}. We first show that $\rho^\bullet(U(\PP^2)) = D_\eta = \{([F^\times],2)\}$, where $U(\PP^2)\subseteq \Spc(\KK(F))$ is the complement of $\supp(\PP^2)$ and $D_\eta$ is notation for the basic open set of homogeneous prime ideals not containing $\eta$.  Indeed, since $\Sigma \PP^2 \simeq C\eta$, the cone of $\eta$, we have that the inclusion $U(\PP^2)\cong \Spc(\KK(F)/\langle \PP^2\rangle)\to \Spc(\KK(F))$ factors through $\Spc(\eta^{-1}\KK(F))$.  Recall from \cite[Corollary 3.11]{Morel:LNM} that $\eta^{-1}\kmw*F\cong W(F)[\eta,\eta^{-1}]$, the ring of Laurent polynomials in a variable $\eta$ over the Witt ring $W(F)$.  We therefore have that
$(~)^0:\Spec^h(\eta^{-1}\kmw*F)\cong \Spec(W(F))\cong \Spec(\FF_2)$ where the final identification holds via \cite{lorenzleicht} because $F$ has no orderings.  Since the non-graded analogue $\rho$ of  $\rho^\bullet$ is surjective (by \cite[Theorem 7.13]{balmer:sss} since $\eta^{-1}\KK(F)$ is connective), we conclude that $\rho^\bullet:\Spc(\eta^{-1}\KK(F))\to \Spec^h(\eta^{-1}\kmw*F)$ is surjective.  The image of $\Spec^h(\eta^{-1}\kmw*F)$ in $\Spec^h(\kmw*F)$ is precisely $D_\eta$.  We conclude that $\rho^\bullet(U(\PP^2)) = D_\eta$.

We now check that $\rho^\bullet(\supp(\PP^2))$ is precisely the (closed) complement $V_\eta = \{\pp\mid \eta\in\pp\}$ of $D_\eta$.  We have that $\rho^\bullet(\supp(\PP^2))\subseteq V_\eta$ by definition of $\rho^\bullet$ and $\supp$.  Note that by \cite[Remark 10.2]{balmer:sss}, $\Spec \ZZ\cong \Spec GW(F)$ via the assignment $p\ZZ\mapsto \dim^{-1} p\ZZ$ for $\dim:GW(F)\to \ZZ$ the dimension homomorphism.  Meanwhile $V_\eta = \{([F^\times],p,\eta)\mid p\text{ prime or }0\}$ and it is elementary to observe that $([F^\times],p,\eta)^0 = \dim^{-1} p\ZZ$.  
Thus $(~)^0:V_\eta\cong \Spec GW(F)\cong \Spec \ZZ$, and the composite $\rho:\Spc(\KK(F))\to \Spec\ZZ$ is surjective by \cite[Theorem 7.13]{balmer:sss}.
It follows that $\rho^\bullet(\supp(\PP^2))\subseteq V_\eta$ consists of at least every prime in $V_{\eta}$ except possibly $([F^\times],2,\eta)$.
Thus it remains to show that $([F^\times],2,\eta)\in\rho^\bullet(\supp(\PP^2))$. For this it suffices to show that the closure of $U(\PP^2)$ intersects $\supp(\PP^2)$ nontrivially. Indeed, by continuity,
$\rho^{\bullet}(\overline{U(\PP^2)} \cap \supp(\PP^2))$ is contained in 
$\overline{\rho^\bullet(U(\PP^2))}\cap V_{\eta} = \{([F^\times],2,\eta)\}$.

Suppose for contradiction that $\Spc(\KK(F))$ is separated by $\overline{U(\PP^2)}$ and $\supp(\PP^2)$. Let $p$ denote the exponential characteristic of $F$ and consider the functor $i:\KK(F)\to \KK(F)[\frac{1}{p}]$ and the induced continuous map
$$
\phi:=\Spc(i):\Spc( \KK(F)[1/p]) \to  \Spc(\KK(F)).
$$ 
Now we have that $\phi^{-1}(\supp(\PP^2)) = \supp(\PP^2[\frac{1}{p}])$ and 
$\phi^{-1}(U(\PP^2)) = U(\PP^2[\frac{1}{p}])$. It follows that $\Spc( \KK(F)[\frac{1}{p}])$ is separated by $\overline{U(\PP^2[\frac{1}{p}])}$ and $\supp(\PP^2[\frac{1}{p}])$. Now by \cite[Corollary B.2]{LYZ} every smooth $F$-scheme is dualizable in $\SHA(F)[\frac{1}{p}]$ and so by 
\cite[Theorem 2.1.3]{HPS:axiomatic}, every object in $\KK(F)[\frac{1}{p}]$ is dualizable. We may thus apply \cite[Theorem 2.11]{balmer:supp} to conclude that the unit of $\KK(F)[\frac{1}{p}]$ decomposes as a nontrivial sum of two objects. It follows that $GW(F)[\frac{1}{p}]$, the endomorphism ring of the unit in $\KK(F)$, decomposes as the product of two nontrivial rings.  
Since $\Spec(GW(F)[\frac{1}{p}])\cong \Spec(\ZZ[\frac{1}{p}])$ has a unique minimal element, $GW(F)[\frac{1}{p}]$ is irreducible, and we have reached a contradiction.  We conclude that $\overline{U(\PP^2)}\cap \supp(\PP^2)\ne \varnothing$, and hence $\rho^\bullet(\supp(\PP^2)) = V_\eta$.  It follows that $\rho^\bullet$ is surjective.

\end{proof}

We can now prove our main theorem.

\begin{proof}[Proof of \aref{thm:main}]
The proof proceeds in four steps.

\subsection*{Step 1}  We first show that it suffices to check surjectivity of $\rho^\bullet$ when $F$ is algebraically or real closed.  For an arbitrary $F$, let $i:F\hookrightarrow F_\alpha$ be a real or algebraic closure of $F$.  Consider the induced functor $i^*:\KK(F)\to \KK(F_\alpha)$, which is tensor triangulated such that $i^*((\AA^1\setminus 0)_F) = (\AA^1\setminus 0)_{F_\alpha}$.   Hence we get a commutative square
\[
\xymatrix{\Spc(\KK(F_\alpha))\ar[r]^{\Spc(i^*)}\ar[d]_{\rho^\bullet} &\Spc(\KK(F))\ar[d]^{\rho^\bullet}\\
\Spec^h(\kmw*{F_\alpha})\ar[r] &\Spec^h(\kmw*F).}
\]
By \aref{thm:thornton}, the sum of the bottom horizontal maps over all $\alpha\in X_F^* = X_F\amalg \{\infty\}$ (where $F_\infty =  \overline{F}$ and $F_\alpha$ is the real closure of $F$ corresponding to $\alpha \in X_F$) is surjective.  Thus if $\rho^\bullet$ is surjective for all algebraically or real closed fields, we may conclude that $\rho^\bullet$ is surjective in general.

\subsection*{Step 2}  We now check that $\rho^\bullet$ is surjective when $F$ is algebraically closed.  This follows from \aref{prop:nonreal} because algebraically closed fields are nonreal.

\subsection*{Step 3}  In this step, we will compute $[-1]^{-1}\kmw*F$ for $F$ real closed.  We use this computation in Step 4 as part of our proof that $\rho^\bullet$ is surjective when $F$ is real closed.

For the moment, take $F$ to be any field of characteristic different from $2$.  Via the twisted logarithm relation, we have
\[
  0 = [1] = 2[-1]+[-1]^2\eta
\]
in $\kmw*F$.  Hence $h = 2+[-1]\eta = 0$ in $[-1]^{-1}\kmw*F$.  It follows that $[-1]^{-1}\kmw*F = [-1]^{-1}(\kmw*F/h)$.  The ring $K^W_*(F) := \kmw*F/h$ also goes by the name Witt $K$-theory.  By \cite[Theorem 6.4.7]{morel:t}, we have $K^W_*(F)\cong W(F)[I\eta^{-1},\eta]$, the extended Rees ring of the Witt ring of $F$ with respect to the fundamental ideal $I$.  This is a graded ring in which $W(F)[I\eta^{-1},\eta]_n = I^n\cdot\{\eta^{-n}\}$, for $I^n$ the $n$-th power of the fundamental ideal in $W(F)$ (where $I^n = W(F)$ for $n\le 0$).  The image of $[-1]$ in $W[I\eta^{-1},\eta]$ is $2\cdot \eta^{-1}$.

Now suppose that $F$ is real closed, in which case $W(F)\cong \ZZ$ via the signature homomorphism and
\[
  K^W_*(F)\cong \ZZ[\eta,[-1]]/(2+[-1]\eta).
\]
Hence
\[
  [-1]^{-1}\kmw*F\cong [-1]^{-1}K^W_*(F)\cong \ZZ[[-1],[-1]^{-1}]\text{,}
\]
\emph{i.e.}, Laurent polynomials in the variable $[-1]$ over $\ZZ$.

\subsection*{Step 4}  We now show that $\rho^\bullet$ is surjective when $F$ is real closed, completing our proof.  Write
\[
  \Spec^h(\kmw*F) = V_{[-1]}\cup D_{[-1]}
\]
where $V_{[-1]}$ is the subspace of primes of the form (1), (2), or (3), and $D_{[-1]}$ is its complement.  Note that $V_{[-1]}$ is the closed set of primes containing $[-1]$, and $D_{[-1]}$ is the basic open of primes not containing $[-1]$.  Base change to $\overline{F}$ and Step 2 imply that $V_{[-1]}$ is hit by $\rho^\bullet$.  To show that $D_{[-1]}$ is hit by $\rho^\bullet$, consider the diagram
\[
\xymatrix{
  \Spc(\KK(F)/\langle C[-1]\rangle)\ar[r]\ar[d] &\Spc([-1]^{-1}\KK(F))\ar[r]\ar[d] &\Spc(\KK(F))\ar[d]\\
  D_{[-1]}\ar[r] &\Spec^h([-1]^{-1}\kmw*F)\ar[r]\ar[d]^{(~)^0} &\Spec^h(\kmw*F)\\
  &\Spec([-1]^{-1}\kmw0F)&{\phantom{\Spec^h(\kmw*F)}}.
}\]
Here $C[-1]$ is the cone of $[-1]:S^0\to S^\alpha$ and $\KK(F)/\langle C[-1]\rangle$ is the Verdier quotient by the thick subcategory $\langle C[-1]\rangle$ generated by $C[-1]$, whence $\Spc(\KK(F)/\langle C[-1]\rangle)$ is naturally homeomorphic to $U(C[-1])$.  Therefore the restriction of $\rho^\bullet$ to $\Spc(\KK(F)/\langle C[-1]\rangle)$ has image inside of $D_{[-1]}$, and this specifies the left-hand vertical map in the diagram.  Since $[-1]$ is invertible in $\KK(F)/\langle C[-1]\rangle$, the quotient functor $\KK(F)\to \KK(F)/\langle C[-1]\rangle$ factors through $[-1]^{-1}\KK(F)$.  The top row of the diagram is given by applying $\Spc$ to this factorization.  Naturality of $\rho^\bullet$ produces the rest of the diagram except for the map $(~)^0:\Spec^h([-1]^{-1}\kmw*F)\to \Spec([-1]^{-1}\kmw0F)$, which is the usual intersection-with-the-$0$-th-graded-piece map.

By Step 3, $(~)^0$ is a homeomorphism.  Additionally, the vertical composite is the comparison map $\rho$, which is surjective by connectivity of the category $[-1]^{-1}\KK(F)$ and \cite[Theorem 7.13]{balmer:sss}.  We conclude that the middle upper vertical $\rho^{\bullet}$ is surjective.  Furthermore, the bottom row in our diagram is a homeomorphism followed by an embedding so that the composite is the inclusion of $D_{[-1]}$ in $\Spec^h(\kmw*F)$.  It follows that the image of the right-hand vertical $\rho^\bullet$ contains all of $D_{[-1]}$.  We conclude that $\rho^\bullet$ is surjective, as desired.
\end{proof}

\section{Explicit primes}\label{sec:explicit}

The methods of the previous section show that $\rho^\bullet$ is surjective without explicitly identifying any tensor triangular primes living over elements of $\Spec^h(\kmw*F)$.  In this section, we construct elements of $\Spc(\Cc(F))$ living over elements of $\Spec^h(\kmw*F)$.  We focus on $\Spc(\Cc(F))$ because we are able to construct a rich class of fields for $\Cc(F)$.  
Using a different approach in \aref{sec:fields}, we  construct fields for $\Spc(\SHA(F)^c)$.
  
Before constructing field spectra, we will need some preliminary facts about fields in tensor triangulated categories, Milnor-Witt $K$-theory, Morel's homotopy $t$-structure, and Dugger-Isaksen's cellularization functor.  We gather these facts in the next four subsections.  The reader familiar with this content should focus on \aref{defn:field}, \aref{prop:ker}, \aref{defn:res}, and \aref{prop:res}, familiarize herself with our notation, and then move on to \aref{subsec:fields}.

\subsection{Fields in tensor triangulated categories}\label{subsec:ttfields}
Let $(\Cc,\otimes,1)$ be a tensor triangulated category.  Recall that a \emph{ring} in $\Cc$ is an object $R$ equipped with a unit $\eta:1\to R$ and multiplication $\mu:R\otimes R\to R$ making $R$ a monoid in $\Cc$.\footnote{The monoid/ring nomenclature is somewhat at odds.  Of course, a monoid in $\ZZ$-modules is a ring, and this is the inspiration for the terminology.}  An \emph{$R$-module} in $\Cc$ is an object $M$ equipped with an action map $\alpha:R\otimes M\to M$ satisfying the standard axioms.  Let $\Pic(\Cc)$ denote the Picard group of $\otimes$-invertible objects in $\Cc$.  We call an $R$-module \emph{free} if it is module-isomorphic to a coproduct of modules of the form $P\otimes R$ where each $P$ is in $\Pic(\Cc)$.

\begin{defn}\label{defn:field}
A ring $K\in \Cc$ is called a \emph{skew field} 
if every $K$-module in $\Cc$ is free.
\end{defn}
 
We will abuse terminology and call a skew field $K$  a \emph{field}, even when $K$ is not commutative. Field objects are important in tensor triangular geometry because they are a rich source of prime ideals.
 
\begin{prop}\label{prop:ker}
Suppose that $K\in \Cc$ is a field.  Let $\Dd$ be a tensor triangulated subcategory of $\Cc$.  Then
\[
  \ker_\Dd(K) := \{X\in \Dd\mid K\otimes X \cong 0\}
\]
is a tensor triangular prime in $\Dd$.
\end{prop}
\begin{proof}
It is easy to check that $\ker_\Dd K$ is a tensor triangular ideal, so all that remains to show is primality.  Suppose that $X\otimes Y\in \ker_\Dd K$ for some $X$, $Y$ in $\Dd$.  Since $K\otimes X$ is a $K$-module and $K$ is a field, we have
\[
  K\otimes X \cong \bigoplus_{i\in I} P_i\otimes K
\]
where $P_i\in \Pic(\Cc)$ for $i\in I$.  Thus
\[
  0\cong K\otimes X\otimes Y\cong \bigoplus_{i\in I}P_i\otimes K\otimes Y.
\]
Since each $P_i$ is $\otimes$-invertible, this is only possible if $I=\varnothing$ (in which case $X\in \ker_\Dd K$) or $K\otimes Y\cong 0$ (in which case $Y\in \ker_\Dd K$).  Thus $\ker_\Dd K\in \Spc(\Dd)$, as desired.
\end{proof}

\begin{remark}
 We are considering $P\otimes R$ to be a free $R$-module for any $P\in \Pic(\Cc)$. In specific cases, one might wish to restrict this notion. For example in the motivic setting (where the structure of $\Pic(\SHA(F))$ is complicated and not completely determined, see \cite{b:pic,hu:pic}), it could be reasonable to only allow smashing with $S^{m+n\alpha}$ rather than arbitrary invertible spectra. (Indeed, the fields we produce are of this form.) However, it will not make a difference in our work below and so we work with the general notion. 
\end{remark}

\subsection{The residue fields of Milnor-Witt $K$-theory}\label{subsec:KMW}

In order to produce fields for $\Cc(F)$, we will need to understand the residue fields of Milnor-Witt $K$-theory, as defined below.

\begin{defn}\label{defn:res}
For $\pp$ a homogeneous prime ideal of $\kmw*F$, we define the \emph{residue field} of $\kmw*F$ at $\pp$,
\[
  L(\pp) := \GrFrac(\kmw*F/\pp),
\]
to be the graded field of fractions of the $\ZZ$-graded integral domain $\kmw*F/\pp$; this is formed by inverting all nonzero homogeneous classes in $\kmw*F/\pp$.
\end{defn}

By elementary computation and \aref{thm:thornton}, we arrive at the following list of residue fields for $\kmw*F$.

\begin{prop}\label{prop:res}
The residue fields of $\kmw*F$ take the following forms (with the obvious $\kmw*F$-module structure):
\begin{enumerate}[(1)]
\item $L([F^\times],\eta)\cong \QQ$,
\item $L([F^\times],\eta,p)\cong \FF_p$,
\item $L([F^\times],2)\cong \FF_2[\eta^{\pm 1}]$,
\item $L([P_\alpha],h)\cong \QQ[\eta^{\pm 1}]$,
\item $L([P_\alpha],\eta,2)\cong \FF_2[[-1]^{\pm 1}]$, and
\item $L([P_\alpha],h,p)\cong \FF_p[\eta^{\pm 1}]$ for $p\ne 2$.
\end{enumerate}
\hfill\qedsymbol
\end{prop}

\subsection{The homotopy $t$-structure}

Morel \cite{morel:t} has introduced a $t$-structure on $\SHA(F)$ defined in the following fashion.  Let $\SHAt{\ge 0}(F)$ denote the full subcategory of \emph{connective motivic spectra}, those $X$ such that $\ul\pi_m(X)_n = 0$ for $m < 0$, $n\in \ZZ$.  Similarly, let $\SHAt{\le 0}(F)$ consist of $X$ such that $\ul\pi_m(X)_n = 0$ for $m> 0$, $n\in \ZZ$.  By \cite[Theorem 5.2.3]{morel:t}, the triple $(\SHA(F),\SHAt{\ge 0}(F),\SHAt{\le 0}(F))$ is a non-degenerate $t$-structure, called the \emph{homotopy $t$-structure}.

In the usual fashion, we can define $\SHAt{\le m}(F)$ and $\SHAt{\ge m}(F)$ for $m\in \ZZ$ and the adjoints to inclusion $\tau_{\le m}:\SHA(F)\to \SHAt{\le m}(F)$ and $\tau_{\ge m}:\SHA(F)\to \SHAt{\ge m}(F)$.  For $X\in \SHA(F)$ we define 
$$
\HH_m X := \tau_{\le m}\tau_{\ge m}X
$$ 
and call $\HH_0 X$ the \emph{heart} of $X$.  Morel identifies the heart $\SHAt{\le 0}(F)\cap \SHAt{\ge 0}(F)$ of the homotopy $t$-structure with the so-called category of \emph{homotopy modules}\footnote{These are $\ul K^{MW}_*(F)$-modules $\ul M_*$ equipped with \emph{contraction isomorphisms} $(\ul M_{n+1})_{-1}\cong \ul M_n$ where $(~)_{-1}$ is Morel's contraction construction \cite[Definition 4.3.10]{morel:t}.} but we will not use this result here.

The following lemma follows from \cite[\S2.3]{grso}.  We include a proof for completeness.

\begin{lemma}\label{lemma:heart}
Let $E$ be a motivic ring spectrum.  Then $\HH_0 E$ is a motivic ring spectrum as well.
\end{lemma}
\begin{proof}
The multiplication on $\HH_0E$ arises as a composite
\[
  \HH_0E\wedge \HH_0E\to \HH_0(\HH_0E\wedge \HH_0E)\to \HH_0E.
\]
We begin by constructing the second map.  To do so, first observe that $\HH_0(\HH_0E\wedge \HH_0E) \simeq \HH_0(\tau_{\ge 0}E\wedge \tau_{\ge 0}E)$.  Then $\HH_0(\HH_0E\wedge \HH_0E)\to \HH_0E$ arises by applying $\HH_0$ to the natural map $\tau_{\ge 0}E\wedge \tau_{\ge 0}E\to \tau_{\ge 0}E$.

To construct $\HH_0E\wedge \HH_0E\to \HH_0(\HH_0E\wedge \HH_0E)$, simply observe that whenever $A$ is connective, there is a natural map $A\to \HH_0A$ (which is just the natural map $A\to \tau_{\le 0}A$).

It is formal to check that the above map $\HH_0E\wedge \HH_0E\to \HH_0E$ together with the unit
\[
  S^0\to \HH_0S^0\to \HH_0E
\]
make $\HH_0E$ a monoid in $\SHA(F)$, as desired.
\end{proof}

\subsection{Cellularization}
Write $i:\SHA(F)^{\cell}\subseteq \SHA(F)$ for the homotopy category of motivic cell spectra, \emph{i.e.}~the localizing subcategory generated by $S^{m+n\alpha}$, for $m,n\in \Z$.
Recall from \cite[Proposition 7.3]{di:cell} that, on the level of homotopy categories, there is a functor 
$$
\Cell:\SHA(F)\to \SHA(F)^{\cell}
$$ 
which is right adjoint to the inclusion. As is standard, we will abuse notation and write again $\Cell(X)$ for $i\Cell(X)$. In particular, we have a natural map $\Cell(X) \to X$. Note also that $\Cell$ is a lax symmetric monoidal functor, as it is right adjoint to a symmetric monoidal functor. We immediately conclude the following.
 
\begin{lemma}\label{lemma:cell}
Let $E$ be a motivic ring spectrum.  Then $\Cell(E)$ is a motivic ring spectrum as well.\hfill\qedsymbol
\end{lemma}

\subsection{Field spectra and explicit primes in $\Spc(\Cc(F))$}\label{subsec:fields}
We now build some cellular field spectra and realize some explicit tensor triangular primes in $\Cc(F)$. We write 
$$
\HH_0^{\cell}(E) := \Cell(\HH_0(E)).
$$

\begin{lemma}\label{lemma:realize}
Fix $\pp\in \Spec^h(\kmw*F)$ and suppose there exists a (homotopy associative) ring spectrum $E$ such that $\pi_0(E)_*\cong L(\pp)$ as $\kmw*F$-modules.  
Then $\HH_0^{\cell}(E)$ is a field in $\SHA(F)^{\cell}$ and 
 $\rho^{\bullet}\bigl(\ker_{\Cc(F)}(\HH_0^{\cell}(E))\bigr) = \pp$.
\end{lemma}
\begin{proof}
First, note that \aref{lemma:heart} and \aref{lemma:cell} imply that $\HH_0^{\cell}(E)$ is a cellular ring spectrum.

Let $M$ be a $\HH_0^{\cell}(E)$-module and note that $\pi_*(M)_*$ is an $L(\pp)$-module.  Since $L(\pp)$ is a graded field, we may choose a homogeneous basis for $\pi_*(M)_*$ as a free $\pi_*(\HH_0^{\cell}(E))_*$-module.  This basis induces a map
\[
  \bigvee_{\lambda\in \Lambda} \Sigma^\lambda \HH_0^{\cell}(E)\to M
\]
where each $\lambda$ is in $\ZZ\oplus \ZZ\{\alpha\}$ which induces an isomorphism on homotopy groups. This map is thus a weak equivalence since homotopy groups detect weak equivalences between cellular motivic spectra.  Hence $M$ is free, as desired, so $\HH_0^{\cell}(E)$ is a cellular motivic field spectrum.

By \aref{prop:ker}, $\ker_{\Cc(F)}(\HH_0^{\cell}E)$ is in $\Spc(\Cc(F))$. Moreover, 
\[
\begin{aligned}
  \rho^{\bullet}\bigl(\ker_{\Cc(F)}(\HH_0^{\cell}E)\bigr) &= 
  \{f\in \kmw*F\mid \cone(f)\notin \ker_{\Cc(F)}(\HH_0^{\cell}(E) ) \}\\
  &= \{f\in\kmw*F\mid \HH_0^{\cell}(E)\wedge \cone(f)\not\simeq *\}\\
  &= \{f\in\kmw*F\mid \HH_0^{\cell}(E)\xrightarrow{f}S^{*\alpha}\wedge\HH_0^{\cell}(E)\text{ is not a w.e.}\}\\
  &= \{f\in \kmw*F\mid L(\pp)\xrightarrow{\cdot f} L(\pp)\text{ is not an iso}\}\\
  &= \pp.
\end{aligned}
\]
\end{proof}
 
\begin{prop}\label{prop:construct}
If $\pp\in \Spec^h(\kmw*F)$ is of type (3), (4), (5), or (6) in \aref{thm:thornton}, then there exists a motivic ring spectrum $E\in \SHA(F)$ such that $\pi_0(E)_*\cong L(\pp)$.
\end{prop}
\begin{proof}
First suppose that $\pp = ([F^\times],2)$ is of type (3) so that $L(\pp)\cong \FF_2[\eta^{\pm 1}]$.  Fix an embedding $i:F\hookrightarrow \overline{F}$ of $F$ into an algebraic closure, $\overline{F}$. The functor 
$i^*:\SHA(F)\to \SHA(\overline{F})$ has a right adjoint, $i_*$, which is lax-symmetric monoidal. In particular, $i_*$ preserves motivic ring spectra. 
Let $KT_{\overline{F}}$ denote Witt $K$-theory \cite{hornbostel:rep} over $\overline{F}$.  Then we may take $E = i_*KT_{\overline{F}}$.  Indeed, we may compute
\[
\begin{aligned}
  \pi_0(E)_* &\cong [S^0,S^{*\alpha}\wedge i_*KT_{\overline{F}}]_{\SHA(F)}\\
  &\cong [S^{-*\alpha},i_*KT_{\overline{F}}]_{\SHA(F)}\\
  &\cong [S^{-*\alpha},KT_{\overline{F}}]_{\SHA(\overline{F})}\\
  &\cong W(\overline{F})[\eta^{\pm 1}]\\
  &\cong \FF_2[\eta^{\pm 1}]
\end{aligned}
\]
since the Witt ring of an algebraically closed field is always isomorphic to $\FF_2$ via the rank homomorphism.

Now suppose that $F$ is formally real with ordering $\alpha$.  (This is necessary for primes of type (4)-(6) to exist.)  Fix an embedding $j:F\hookrightarrow F_\alpha$.  Suppose that $\pp$ is of type (4) so that $L(\pp)\cong \QQ[\eta^{\pm 1}]$.  Then we may take 
$E = j_* KT_{F_\alpha}\otimes \QQ$ and a computation similar to that above gives
\[
  \pi_0(E)_*\cong \QQ[\eta^{\pm 1}]
\]
since $W(F_\alpha)\cong \ZZ$.

Suppose that $\pp$ is of type (6) so that $L(\pp)\cong \FF_p[\eta^{\pm 1}]$.  Then we may take $E = j_* KT_{F_\alpha}/p$, and we again get the desired isomorphism
\[
  \pi_0(E)_*\cong \FF_p[\eta^{\pm 1}].
\]

Now suppose $\pp = ([P_\alpha],\eta,2)$ is of type (5) so that $L(\pp)\cong \FF_2[[-1]^{\pm 1}]$.  Let $(H\FF_2)_{F_\alpha}$ denote the mod $2$ motivic cohomology spectrum over $F_\alpha$.  
Let $[-1]^{-1}(H\FF_2)_{F_\alpha}$ denote the colimit of iterated multiplication by $[-1]$ on $(H\FF_2)_{F_\alpha}$.  Then we may take 
$E = j_* [-1]^{-1}(H\FF_2)_{F_\alpha}$.  Indeed, we compute
\[
\begin{aligned}
  \pi_0(E)_* &\cong [S^0,S^{*\alpha}\wedge j_* [-1]^{-1}(H\FF_2)_{F_\alpha}]_{\SHA(F)}\\
  &\cong [S^{-*\alpha},j_*[-1]^{-1}(H\FF_2)_{F_\alpha}]_{\SHA(F)}\\
  &\cong [S^{-*\alpha},[-1]^{-1}(H\FF_2)_{F_\alpha}]_{\SHA(F_\alpha)}\\
  &\cong [-1]^{-1}K^M_*(F_\alpha)\\
  &\cong \FF_2[[-1]^{\pm 1}].
\end{aligned}
\]
The penultimate isomorphism utilizes the Milnor conjecture \cite{voe:mod2}, while the final isomorphism is an invocation of \cite[Example 1.6]{milnor:quadK}.
\end{proof}

Combining \aref{lemma:realize} and \aref{prop:construct}, we arrive at the following list of explicit primes in $\Cc(F)$ and their images in $\Spec^h(\kmw*F)$.

\begin{theorem}\label{thm:explicit}
Fix an embedding $i:F\hookrightarrow \overline{F}$ of $F$ into its algebraic closure.  If $\alpha\in X_F\ne \varnothing$, fix an embedding $j = j_\alpha:F\hookrightarrow F_\alpha$ of $F$ into its real closure with respect to $\alpha$.  Let $E_3 = i_*KT_{\overline{F}}$, $E_4 = j_*KT_{F_\alpha}\otimes \QQ$, $E_5 = j_*[-1]^{-1}(H\FF_2)_{F_\alpha}$, and $E_6 = j_*KT_{F_\alpha}/p$.  
Writing 
$$
\Pp_{i} =  \ker_{\Cc(F)}(\HH_0^{\cell}(E_i)),
$$ 
we have
\[
\begin{aligned}
  \rho^{\bullet}(\Pp_{3}) &= ([F^\times],2)\text{,}\\
  \rho^{\bullet}(\Pp_{4}) &= ([P_\alpha],h)\text{,}\\
 \rho^{\bullet}(\Pp_{5}) &= ([P_\alpha],\eta,2)\text{,}\\
\rho^{\bullet}(\Pp_{6}) &= ([P_\alpha],h,p)\text{.}
\end{aligned}
\]\hfill\qedsymbol
\end{theorem}
 
\begin{remark}\label{rmk:cellfields}
In \aref{prop:topmot} we will determine explicit tensor triangular primes living over primes in $\kmw*F$ of type (1) and (2).  These are realized as acyclics for field spectra in \aref{rmk:kersing}.
\end{remark}

\section{Betti realization and tensor triangular primes}\label{sec:betti}
In this section, we study the role of Betti realization functors in determining tensor triangular primes in $\SHA(F)^c$.  We consider both the complex Betti realization functor
\[
  \RE_{B,i}:\SHA(F)^c\longrightarrow \SH^{\mathrm{fin}}
\]
associated with a complex embedding $i:F\hookrightarrow \CC$ and, for each $\alpha\in X_F$, the $C_2$-equivariant Betti realization functor
\[
  \RE_{B,j}^{C_2}:\SHA(F)^c\longrightarrow \SH(C_2)^c
\]
associated with an order-preserving embedding $j:F\hookrightarrow \RR$.  See \cite[\S4.4]{ho:gal} for a review of these constructions.  We will frequently leave the embedding into $\CC$ or $\RR$ implicit and call these functors $\RE_B$ and $\RE_B^{C_2}$, respectively.

In order to study $\Spc(\RE_B)$ and $\Spc(\RE_B^{C_2})$, we will need to know the structure of the spaces $\Spc(\SH^{\fin})$ and $\Spc(\SH(C_2)^c)$, respectively.  We recall these spaces in the next two subsections, following Balmer \cite{balmer:sss} and Balmer-Sanders \cite{balmersanders}, respectively.  In the final subsection, we present our results on $\Spc(\RE_B)$ and $\Spc(\RE_B^{C_2})$.

\subsection{Tensor triangular primes in $\SH^{\fin}$}
Building on \cite{hs:nil2}, Balmer determines the structure of $\Spc(\SH^{\fin})$ in \cite[\S 9]{balmer:sss}.  We briefly recall the result here.

Let $p$ be a rational prime and let $n$ be a positive integer.  Let $K(p,n-1)$ denote $(n-1)$-th Morava $K$-theory at the prime $p$, with the convention that $K(p,0) = H\QQ$, the rational Eilenberg-MacLane spectrum.  Let $\Cc_{p,n}$ denote the kernel of $K(p,n-1)$-homology.\footnote{Balmer sets $\Cc_{p,n}$ equal to the kernel of $K(p,n-1)\wedge (~)_{(p)}$.  Since $K(p,n-1)$ is already $p$-local and $p$-localization is smashing, we have $\ker K(p,n-1) = \ker K(p,n-1)\wedge (~)_{(p)}$.}  In particular, $\Cc_{p,1} = \Cc_{q,1}$ for all $p$, $q$, and we denote this set $\Cc_{0,1}$.  Let $\Cc_{p,\infty}$ denote the set of $p$-locally trivial spectra.

\begin{theorem}[{\cite[Corollary 9.5]{balmer:sss}}]\label{thm:top}
Every element of $\Spc(\SH^{\fin})$ is of the form $\Cc_{p,n}$ for some rational prime $p$ and $n$ a positive integer or $\infty$.  The only duplication occurs when $n=1$, in which case $\Cc_{p,1} = \Cc_{0,1}$, the set of finite torsion spectra.  The Hasse diagram for $\Spc(\SH^{\fin})$ takes the form shown in \aref{fig:top} and the closure of $S\subseteq \Spc(\SH^{\fin})$ consists of all the elements in or above $S$ in the Hasse diagram.
\end{theorem}

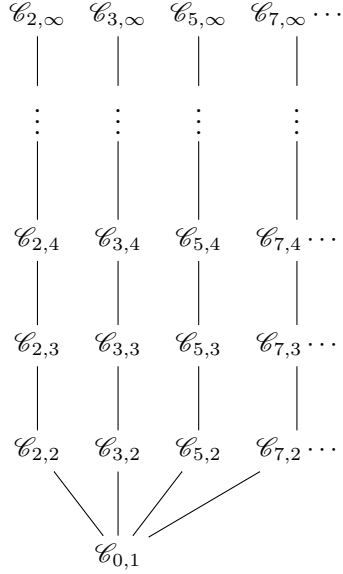
\begin{figure}
\[
\xymatrixcolsep{.25pc}\xymatrix{
  \Cc_{2,\infty}\ar@{-}[d] &\Cc_{3,\infty}\ar@{-}[d] &\Cc_{5,\infty}\ar@{-}[d] &\Cc_{7,\infty}\cdots\ar@{-}[d]\\
  \vdots\ar@{-}[d] &\vdots\ar@{-}[d] &\vdots\ar@{-}[d] &\vdots\ar@{-}[d]\\
  \Cc_{2,4}\ar@{-}[d] &\Cc_{3,4}\ar@{-}[d] &\Cc_{5,4}\ar@{-}[d] &\Cc_{7,4}\cdots\ar@{-}[d]\\
  \Cc_{2,3}\ar@{-}[d] &\Cc_{3,3}\ar@{-}[d] &\Cc_{5,3}\ar@{-}[d] &\Cc_{7,3}\cdots\ar@{-}[d]\\
  \Cc_{2,2}\ar@{-}[dr] &\Cc_{3,2}\ar@{-}[d] &\Cc_{5,2}\ar@{-}[dl] &\Cc_{7,2}\cdots\ar@{-}[dll]\\
  &\Cc_{0,1}
}\]
\caption{The Hasse diagram for $\Spc(\SH^{\fin})$.  Inclusion of tensor triangular primes is downwards, and closure is upwards in the diagram.}\label{fig:top}
\end{figure}  

\subsection{Tensor triangular primes in $\SH(C_2)^c$}
In \cite{balmersanders}, Balmer and Sanders completely determine the tensor triangular spectrum $\Spc(\SH(G)^c)$ for $G$ a finite group of square-free order.  Here we recall their results, specializing to the group $G=C_2$.

Let $p$ be a rational prime or $0$, let $n$ be a positive integer or $\infty$, let $H$ be a subgroup of $C_2$, and let $\Phi^H:\SH(C_2)\to \SH$ denote the geometric fixed points functor for $H$.  Define
\[
  \Pp(H,p,n) = (\Phi^H)^{-1}(\Cc_{p,n}).
\]

\begin{theorem}[{\cite[Theorem 8.12 and Diagram (1.3)]{balmersanders}}]\label{thm:bs}
Every element of $\Spc(\SH(C_2)^c)$ is of the form $\Pp(H,p,n)$ for precisely one $H$, $p$, and $n$.  The Hasse diagram for $\Spc(\SH(C_2)^c)$ takes the form show in \aref{fig:c2} and the closure of $S\subseteq \Spc(\SH(C_2)^c)$ consists of all the elements in or above $S$ in the Hasse diagram.
\end{theorem}

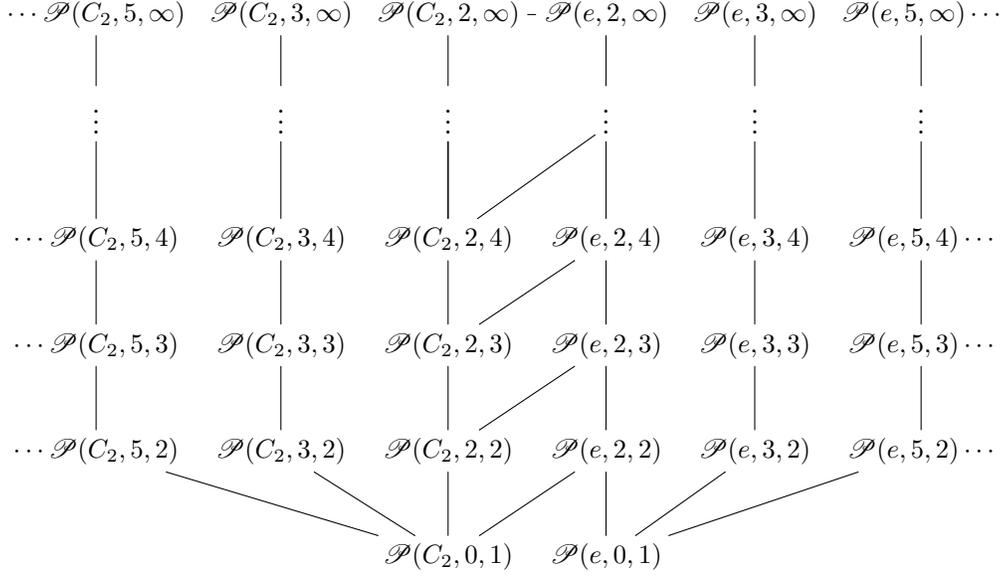
\begin{figure}
\[
\xymatrixcolsep{.25pc}\xymatrix{
\cdots\Pp(C_2,5,\infty)\ar@{-}[d] &\Pp(C_2,3,\infty)\ar@{-}[d] &\Pp(C_2,2,\infty)\ar@{-}[d]\ar@{-}[r] &\Pp(e,2,\infty)\ar@{-}[d] &\Pp(e,3,\infty)\ar@{-}[d] &\Pp(e,5,\infty)\cdots\ar@{-}[d]\\
\vdots\ar@{-}[d] &\vdots\ar@{-}[d] &\vdots\ar@{-}[d]\ar@{-}[d] &\vdots\ar@{-}[d]\ar@{-}[dl]&\vdots\ar@{-}[d]&\vdots\ar@{-}[d]\\
\cdots\Pp(C_2,5,4)\ar@{-}[d] &\Pp(C_2,3,4)\ar@{-}[d] &\Pp(C_2,2,4)\ar@{-}[d] &\Pp(e,2,4)\ar@{-}[dl]\ar@{-}[d] &\Pp(e,3,4)\ar@{-}[d] &\Pp(e,5,4)\cdots\ar@{-}[d]\\
\cdots\Pp(C_2,5,3)\ar@{-}[d] &\Pp(C_2,3,3)\ar@{-}[d] &\Pp(C_2,2,3)\ar@{-}[d] &\Pp(e,2,3)\ar@{-}[dl]\ar@{-}[d] &\Pp(e,3,3)\ar@{-}[d] &\Pp(e,5,3)\cdots\ar@{-}[d]\\
\cdots\Pp(C_2,5,2)\ar@{-}[drr] &\Pp(C_2,3,2)\ar@{-}[dr] &\Pp(C_2,2,2)\ar@{-}[d] &\Pp(e,2,2)\ar@{-}[dl]\ar@{-}[d] &\Pp(e,3,2)\ar@{-}[dl] &\Pp(e,5,2)\cdots\ar@{-}[dll]\\
&&\Pp(C_2,0,1)&\Pp(e,0,1)
}\]
\caption{The Hasse diagram for $\Spc(\SH(C_2)^c)$.  Inclusion of tensor triangular primes is downwards, and closure is upwards in the diagram.}\label{fig:c2}
\end{figure}

\subsection{The maps $\Spc(\RE_B)$ and $\Spc(\RE_B^{C_2})$}
\begin{prop}\label{prop:topmot}
Suppose $F$ embeds into $\CC$ with associated Betti realization functor $\RE_B$.  Then the composite map $\rho^\bullet\circ \Spc(\RE_B):\Spc(\SH^\fin)\to \Spec^h(\kmw*F)$ takes the following values:
\[
\begin{aligned}
  \Cc_{0,1}&\longmapsto ([F^\times],\eta)\text{,}\\
  \Cc_{p,n}&\longmapsto ([F^\times],p,\eta)\text{ for all $p$ and $n\ge 2$.}
\end{aligned}
\]
In particular, the image of $\rho^\bullet\circ \Spc(\RE_B)$ is
\[
  V_\eta\smallsetminus \{([P_\alpha],2,\eta)\mid \alpha\in X_F\}\text{,}
\]
the collection of primes of type (1) or (2).
\end{prop}
\begin{proof}
Let $\pi^*(S^0)$ denote the cohomotopy groups of the sphere so that $\pi^n(S^0) = [S^0,S^n] \cong \pi_{-n}(S^0)$.  Betti realization induces a commutative diagram
\[
\xymatrix{
  \Spc(\SH^{\fin})\ar[r]^{\Spc(\RE_B)}\ar[d]_{\rho^\bullet_{S^1}} &\Spc(\SHA(F)^c)\ar[d]^{\rho^\bullet_{S^\alpha}}\\
  \Spec^h(\pi^*(S^0))\ar[r] &\Spec^h(\kmw*F)
}\]
wherein the bottom horizontal arrow takes $\pi^{>0}(S^0)$ to $([F^\times],\eta)$ and takes $(p,\pi^{>0}(S^0))$ to $([F^\times],p,\eta)$.  The rest of the computation is easy.
\end{proof}

For completeness, we include the following slight enhancement of \cite[Proposition 10.4]{balmer:sss}.

\begin{corollary}\label{cor:monicC}
When $F=\CC$, the map $\Spc(\RE_B)$ is monic with image contained in $\supp(\PP^2)$.
\end{corollary}
\begin{proof}
Let $c^*:\SH^{\fin}\to \SHA(F)^{c}$ be the functor induced by the constant presheaf functor. The map $\Spc(\RE_B)$ is monic because $\RE_B\circ c^* = \id$.  Since $X_F=\varnothing$, the image of the map $\rho^\bullet\circ \Spc(\RE_B)$ is exactly $V_\eta$, and $(\rho^\bullet)^{-1}(V_\eta) = \supp(C\eta) = \supp(\PP^2)$.
\end{proof}

We now turn to the study of $\Spc(\RE_B^{C_2})$ when $F$ has a real embedding.

\begin{prop}\label{prop:eqmot}
Suppose $F$ has an embedding into $\RR$ compatible with $\alpha\in X_F$ and let $\RE_{B,\alpha}^{C_2}$ denote the corresponding $C_2$-equivariant Betti realization functor.  Then the composite map 
$\rho^\bullet\circ \Spc(\RE_{B,\alpha}^{C_2}):\Spc(\SH(C_2)^c)\to \Spec^h(\kmw*F)$ takes the following values:
\[
\begin{aligned}
  \Pp(e,0,1)&\longmapsto ([F^\times],\eta)\text{,}\\
  \Pp(C_2,0,1)&\longmapsto ([P_\alpha],h)\text{,}\\
  \Pp(e,p,n)&\longmapsto ([F^\times],p,\eta)\text{ for all $p$ and $n\ge 2$}\text{,}\\
  \Pp(C_2,p,n)&\longmapsto ([P_\alpha],p,h)\text{ for $p\ne 2$ and $n\ge 2$, and}\\
  \Pp(C_2,2,n)&\longmapsto ([P_\alpha],2,\eta)\text{ for $n\ge 2$.}
\end{aligned}
\]
\end{prop}

\begin{remark}
 If $F$ admits an embedding into $\R$ then for each $\alpha\in X_F$ there is 
 an embedding compatible with the ordering $\alpha$. The proposition implies that the union of the images of all of the resulting 
 $\rho^\bullet\circ \Spc(\RE_{B,\alpha}^{C_2})$ is 
\[
  \Spec^h(\kmw*F)\smallsetminus \{([F^\times],2)\}\text{.}
\]
\end{remark}

\begin{proof}
Let $g = \rho^\bullet\circ \Spc(\RE_{B,\alpha}^{C_2})$.  It is easy to check that $g(\Pp(e,0,1)) = ([F^\times],\eta)$ and $g(C_2,0,1) = ([P_\alpha],h)$.  Since $g$ is continuous, we know that $g(\overline{\Pp(e,0,1)})\subseteq \overline{([F^\times],\eta)}$, where $\overline{(~)}$ denotes the closure operator.  Observing the form of $\Spc(\SH(C_2)^c)$ in \aref{fig:c2} and of $\Spec^h(\kmw*F)$ in \aref{fig:kmw}, we see then that $g(\Pp(e,p,n)) = ([F^\times],p,\eta)$ for all $p$ and $n\ge 2$.

We now check that $g(\Pp(C_2,2,n)) = ([P_\alpha],2,\eta)$ for $n\ge 2$.  Since $g(\Pp(C_2,0,1)) = ([P_\alpha],h)$, the closure continuity condition tells us that $g(\Pp(C_2,2,n))$ is either $([F^\times],2,\eta)$ or $([P_\alpha],2,\eta)$.  Observe, though, that $\Phi^{C_2}\RE_{B,\alpha}^{C_2}C[-1]\simeq *$.  (Simply apply $\Phi^{C_2}\RE_{B,\alpha}^{C_2}$ to the defining cofiber sequence for $C[-1]$.)  Thus 
\begin{align*}
  [-1]\notin \rho^\bullet\bigl(\Spc(\RE_{B,\alpha}^{C_2})&(\Pp(C_2,2,n))\bigr) 
  \\
  &= \{f\in \kmw*F\mid (\Phi^{C_2}\RE_{B,\alpha}^{C_2}Cf)_{(2)}\wedge K(2,n-1)\not\simeq *\}.
\end{align*}
Also $g(\Pp(C_2,2,n)) \neq ([P_\alpha],h)$ since, \emph{e.g.},  $\Phi^{C_2}(C\eta)\wedge K(2,n-1)\not\simeq \ast$ so $\eta\in g(\Pp(C_2,2,n))$.  We conclude that 
 $g(\Pp(C_2,2,n)) = ([P_\alpha],2,\eta)$.

A similar argument shows that $g(C_2,p,n) = ([P_\alpha],p,h)$ for $p$ odd and $n\ge 2$.
\end{proof}

\begin{remark}\label{rmk:suppCeta}
By \aref{thm:main} and the fact that $C\eta\simeq \Sigma\PP^2$, we know that $([F^\times],2)$ belongs to $\rho^\bullet(\supp(\PP^2))$.  This and \aref{prop:eqmot} imply that $\Spc(\RE_B^{C_2})$ does not hit all of $\supp(\PP^2)$, indicating that the tensor triangular spectrum $\Spc(\SHA(\RR)^c)$ is
richer than $\Spc(\SH(C_2)^c)$.  Compare with \aref{cor:monicC}.  The moral seems to be that significant new ideas are necessary to unravel the structure of $\supp(\PP^2)\subseteq \Spc(\SHA(F)^c)$.
\end{remark}

\begin{corollary}\label{cor:monicR}
When $F\subseteq \RR$, the map $\Spc(\RE_B^{C_2})$ is monic with image contained in
\[
  (\rho^\bullet)^{-1}\left(\Spec^h(\kmw*F)\smallsetminus \{([F^\times],2)\}\right)\subsetneq \Spc(\SHA(F)^c).
\]
\end{corollary}
\begin{proof}
Let $c_{F[i]/F}^*$ denote the functor of \cite[Theorem 4.6]{ho:gal}.   The map $\Spc(\RE_B^{C_2})$ is monic because 
$\RE_B^{C_2}\circ c_{F[i]/F}^* = \id$.  The restriction on the image of $\Spc(\RE_B^{C_2})$ follows immediately from \aref{prop:eqmot}.
\end{proof}

\begin{remark}\label{rmk:joa}
This answers a question posed in \cite[Remark 7.2.11]{joa}.  In particular, $\Spc(\RE_B^{C_2})$ is never surjective.
\end{remark}

We conclude this section by providing a sketch of known structures in $\Spc(\SHA(F)^c)$ when $F$ is a subfield of $\RR$ with precisely one ordering.  (For example, $F$ could be $\RR$ itself.)  See \aref{fig:sketch} below.  These structures include the images of the injective maps $\Spc(\RE_B)$ and $\Spc(\RE_B^{C_2})$ along with a mysterious (but nonempty) subspace $(\rho^\bullet)^{-1}([F^\times],2)$.  We have also drawn the map $\rho^\bullet$ with target $\Spec^h(\kmw*F)$ so that known parts of fibers are aligned vertically over their images.  Note that we have drawn $\Spec^h(\kmw*F)$ in a fashion slightly different from that in \aref{fig:kmw} so as to highlight similarities with $\Spc(\SH(C_2)^c)$.

The interested reader can easily elaborate upon \aref{fig:sketch} in the case where $F$ is a subfield of $\RR$ with more than one ordering.  The result would be an ``$X_F$ fan'' over primes in $\kmw*F$ of types (4), (5), and (6).  To see what is known for $F$ a nonreal subfield of $\CC$, simply delete everything of these types.

\begin{remark}
Since the first version of this paper was written, work of Bachmann \cite{bman} appeared which makes it possible to completely specify $U(\cone([-1]))$, the portion of the spectrum $\Spc(\SHA(F)^c)$ living over primes in $\kmw*F$ of types (4), (5), and (6).  Bachmann proves that $\SHA(F)[[-1]^{-1}]$ is equivalent to $\SH(X_F)$, the homotopy category of sheaves of spectra on $X_F$.  It follows $U(\cone([-1]))$ is homeomorphic to $\Spc(\SH(X_F))$.  Thus $U(\cone([-1]))$ is the product of $X_F$ with the topological space specified by the towers on the left-hand side of \aref{fig:sketch}.
\end{remark}

\begin{figure}
\begin{center}
\includegraphics[width=5.5in]{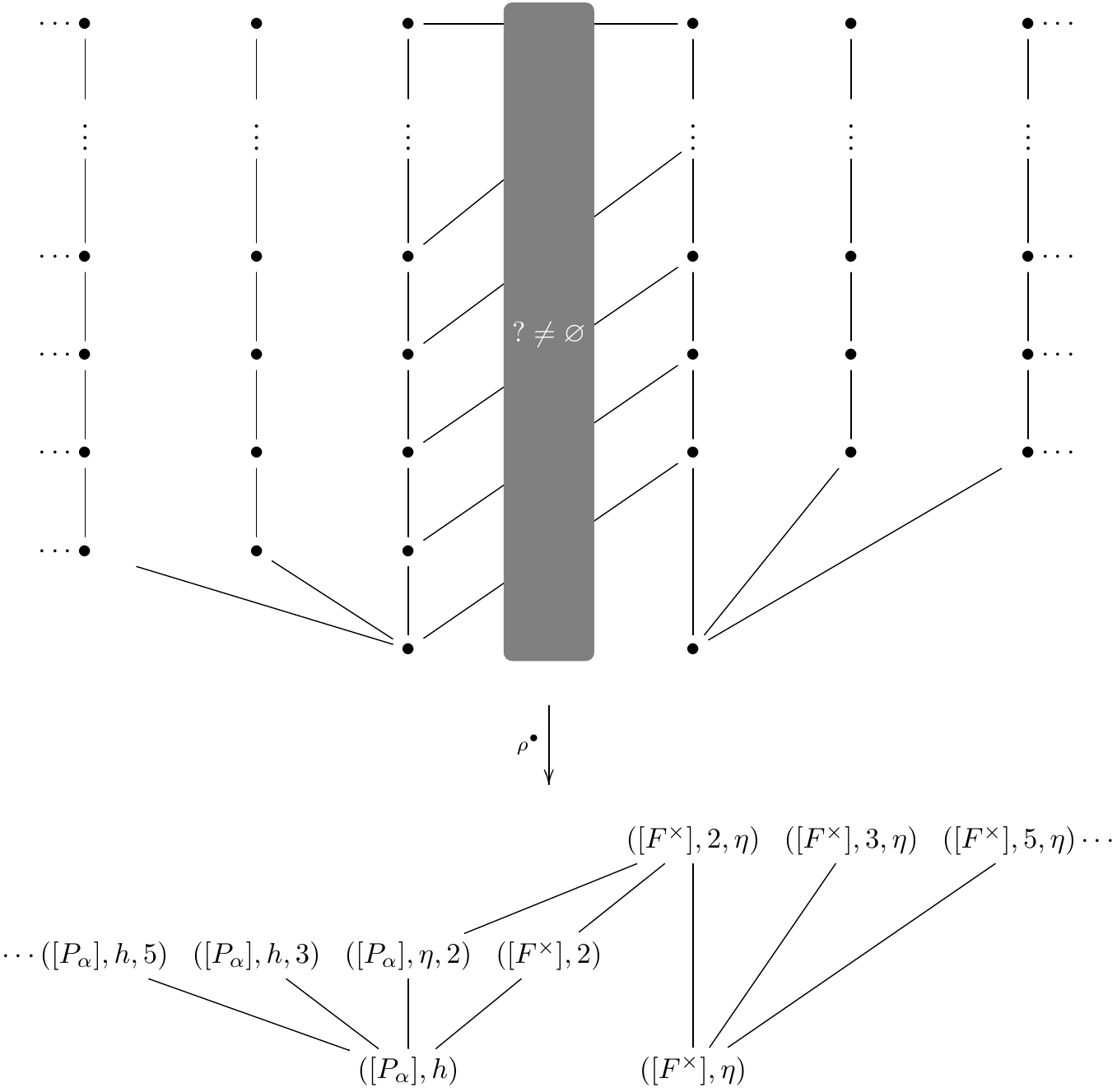}
\end{center}
%\[
%\xymatrixcolsep{.2pc}\xymatrix{
%\cdots\bullet\phantom{\cdots}\ar@{-}[d] &\bullet\ar@{-}[d] &\bullet\ar@{-}[d]\ar@{-}[rr] &?\ne\varnothing &\bullet\ar@{-}[d] &\bullet\ar@{-}[d] &{\phantom{\cdots}}\bullet\cdots\ar@{-}[d]\\
%\vdots\ar@{-}[d] &\vdots\ar@{-}[d] &\vdots\ar@{-}[d] & &\vdots\ar@{-}[d] &\vdots\ar@{-}[d] &\vdots\ar@{-}[d]\\
%\cdots\bullet{\phantom{\cdots}}\ar@{-}[d] &\bullet\ar@{-}[d] &\bullet\ar@{-}[d]\ar@{-}[ur] & &\bullet\ar@{-}[d] &\bullet\ar@{-}[d] &{\phantom{\cdots}}\bullet\cdots\ar@{-}[d]\\
%\cdots\bullet{\phantom{\cdots}}\ar@{-}[d] &\bullet\ar@{-}[d] &\bullet\ar@{-}[d]\ar@{-}[uurr] & &\bullet\ar@{-}[d] &\bullet\ar@{-}[d] &{\phantom{\cdots}}\bullet\cdots\ar@{-}[d]\\
%\cdots\bullet{\phantom{\cdots}}\ar@{-}[d] &\bullet\ar@{-}[d] &\bullet\ar@{-}[d]\ar@{-}[uurr] & &\bullet\ar@{-}[dd] &\bullet\ar@{-}[ddl] &{\phantom{\cdots}}\bullet\cdots\ar@{-}[ddll]\\
%\cdots\bullet{\phantom{\cdots}}\ar@{-}[drr] &\bullet\ar@{-}[dr] &\bullet\ar@{-}[d]\ar@{-}[uurr] & & & &\\
%&&\bullet\ar@{-}[uurr] &\ar@{}[dd]^(.25){}="a"^(.75){}="b" \ar "a";"b"_{\rho^\bullet} &\bullet &&\\
%\save "1,4"."7,4"*[gray]\frm<3pt>{*}*[white]{\txt{\\ \\ \\ \\ \\ \\ \\ \\ \\ \\ \\ \\ \\ \\ \\ \\ \\ \\ \\ \\ \\ \\ $? \ne \varnothing$}} \restore
%\\
%&&&&([F^\times],2,\eta)\ar@{-}[dd] &([F^\times],3,\eta)\ar@{-}[ddl] &([F^\times],5,\eta)\cdots\ar@{-}[ddll]\\
%\cdots([P_\alpha],h,5)\ar@{-}[drr] &([P_\alpha],h,3)\ar@{-}[dr] &([P_\alpha],\eta,2)\ar@{-}[d]\ar@{-}[urr] &([F^\times],2)\ar@{-}[dl]\ar@{-}[ur]\\
%&&([P_\alpha],h)&&([F^\times],\eta)
%}\]
\caption{Known structures in $\Spc(\SHA(F)^c)$ for $F$ a subfield of $\RR$ with one ordering $\alpha$, arranged fiberwise over $\Spec^h(\kmw*F)$ with respect to $\rho^\bullet$.}\label{fig:sketch}
\end{figure}

\section{Field spectra for 
\texorpdfstring{$\SHA(F)^c$}{SH(F)}}\label{sec:fields}

In \aref{sec:explicit}, we constructed fields in the homotopy category of cellular spectra.  Acyclics for these fields gave explicit tensor triangular primes in $\Cc(F)$.  In this section we further explore motivic fields. We begin by showing that fields from the topological stable homotopy category can be imported into $\SHA(F)$, when $F$ admits a complex embedding.

If $F$ is a subfield of $\CC$ then we have an adjoint pair 
$\RE_{B}:\SHA(F) \rightleftarrows \SH : \Sing_{B}$, where $\RE_B$ extends the functor sending a smooth complex variety $X$ to the space $X(\CC)$. Similarly, if $F$ is a subfield of $\RR$, we have an adjoint pair 
$\RE_{B}^{C_2}:\SHA(F) \rightleftarrows \SH(C_2) : \Sing_{B}^{C_2}$. 
In order to treat both of these cases simultaneously, we write 
$$ 
\RE:\SHA(F) \rightleftarrows \SH(G):\Sing
$$
for either of these pairs where $G$ is either $\{e\}$ or $C_2$. 

\begin{lemma}\label{lem:Scop}
 The functor $\Sing:\SH(G)\to \SHA(F)$ commutes with all coproducts.
\end{lemma}
\begin{proof}
 For any $X$ in $\Sm/F$ and any $n\in \Z$, $\Sigma^n_TX_+$ is dualizable by \cite[Theorem 4.9]{RoendOstHZModules} and since $\RE$ is symmetric monoidal $\RE(\Sigma^n_TX_+)$ is dualizable (see e.g., \cite[Proposition 3.10]{HuFauskMay:isos}), hence compact. Using adjunction, it is straightforward to verify that \[[\Sigma^n_TX_+,\bigvee_{\alpha}\Sing(E_{\alpha})]_{\SHA(F)}
 \iso  [\Sigma^n_T X_+,\Sing(\bigvee_{\alpha} E_{\alpha})]_{\SHA(F)}\]
which implies that the map $\bigvee_{\alpha} \Sing(E_{\alpha})\to \Sing\left(\bigvee_{\alpha} E_{\alpha}\right)$ is an equivalence.
 \end{proof}

\begin{lemma}\label{lem:Seqv}
The canonical map 
 $$
 \Sing(K)\wedge X\xrightarrow{\wkeq} \Sing(K\wedge \RE(X))
 $$
 is an equivalence for 
 any $K$ in $\SH(G)$ and any $X$ in $\SHA(F)$. 
\end{lemma}
\begin{proof}
The map of the lemma is the composite of maps
$$
\Sing(K)\wedge X \to \Sing(K)\wedge \Sing\RE(X) \to \Sing (K\wedge \RE(X)).
$$

Consider the full subcategory $\mathcal{A}\subseteq \SHA(F)$ whose objects satisfy the condition that $\Sing(K)\wedge X\wkeq \Sing(K\wedge \RE(X))$. The category $\mathcal{A}$ is triangulated and by the previous lemma it is closed under arbitrary coproducts, \emph{i.e.}, it is a localizing subcategory. It therefore suffices to show that  $\mathcal{A}$ contains every compact spectrum $X$. 
 
 Since $F$ has characteristic zero, every smooth quasi-projective $F$-scheme is dualizable by \cite[Theorem 4.9]{RoendOstHZModules}. 
By \cite[Theorem 2.1.3]{HPS:axiomatic}, this implies that
$\SHA(F)^{c}$ is equal to the category of dualizable objects. Since $\RE$ is strong symmetric monoidal, 
$\RE(DX) \wkeq D(\RE X)$ and $F(Y, \Sing(K))\wkeq \Sing F(\RE(Y), K)$, see e.g. 
\cite[Proposition 3.10 and (3.4)]{HuFauskMay:isos}. We thus have equivalences
\begin{align*}
\Sing(K)\wedge X &  \wkeq F(DX, \Sing(K))\\
& \wkeq \Sing F(\RE(DX), K) \\
& \wkeq \Sing F(D(\RE X), K) \\
& \wkeq\Sing (K\wedge \RE X).
\end{align*}
\end{proof}

 The image of 
$\Spc(\RE):\Spc(\SH(G)^{c})\to \Spc(\SHA(F)^{c})$ is characterized as follows.

\begin{theorem}\label{thm:kerchar}
 Let $K$ be a field in $\SH(G)$. Then 
$\Sing(K)$ is a field in $\SHA(F)$.
Moreover, letting $\ker K\in \Spc(\SH(G)^{\fin})$ be the prime ideal of $K$-acyclics, we have  
$$
\Spc(\RE)(\ker K) = \ker(\Sing(K)).
$$

\end{theorem}
\begin{proof}

The functor $\mathcal{S}$ is lax monoidal as it is right adjoint to a monoidal functor. In particular, if $K$ is a ring in $\SH(G)$, then $\mathcal{S}(K)$ is a ring in $\SHA(F)$.

Let $c^*$ denote the functor from \cite[Theorem 4.6]{ho:gal}.  We first observe that if $X$ is an $\Sing(K)$-module then $\RE(X)$ is a $K$-module. To see this, first note that the adjoint of the canonical isomorphism $\RE c^*(K) \iso K$ is a ring map  $c^*(K)\to \Sing(K)$.  Therefore $X$ is a $c^*(K)$-module and hence $\RE(X)$ is a $K$-module. Since $K$ is a field, we  have that $\RE(X) = \bigvee_{\alpha} K$. It follows from   \aref{lem:Scop} that $\Sing\RE(X) \wkeq \bigvee_{\alpha} \Sing(K)$ and so $\Sing\RE(X)$ is a free $\Sing(K)$-module. 

Now consider the comparison of retracts
$$
\xymatrix{
X \ar[r]\ar[d] & \Sing(K)\wedge X \ar[r]\ar[d]^{\wkeq} & X \ar[d] 
\\
\Sing\RE(X) \ar[r] & \Sing(K\wedge \RE(X)) \ar[r] & \Sing\RE(X).
}
$$
By \aref{lem:Seqv} the middle arrow is an equivalence and therefore so is the outer arrow and so $X$ is a free $\Sing(K)$-module as well. This establishes the first statement.

To establish the second statement, we need to see that
$\Sing(K)\wedge X \wkeq \ast$ if and only if $K\wedge \RE X \wkeq \ast$. This follows from \aref{lem:Seqv} and the observation that if $\Sing(Z)\wkeq \ast$ then $Z\wkeq \ast$ (since $\pi_{n}(Z) = [S^{n},\Sing(Z)]_{\SHA(\CC)}$). 
\end{proof}

\begin{remark}\label{rmk:kersing}
In particular, \aref{thm:kerchar} and \aref{prop:topmot} imply that $\ker(\Sing_B(H\QQ))$ and $\ker(\Sing_B(K(p,n-1)))$ are in $\Spc(\SHA(F))$, mapping to $([F^\times],\eta)$ and $([F^\times],p,\eta)$, respectively, under $\rho^\bullet$.  

Similarly, one can show that $\ker(\Sing_B^{C_2}(EC_{2+}\wedge K(p,n-1)))$ and $\ker(\Sing_B^{C_2}(\widetilde{E}C_2\wedge K(p,n-1))$ belong to $\Spc(\SHA(F))$ with images under $\rho^\bullet$ specified by \aref{prop:eqmot}.
\end{remark}

Lastly we note the following restriction on fields in $\SHA(F)$. 
 
\begin{prop}
Let $F$ be a subfield of $\CC$. If $E$ is a field in $\SHA(F)$ then $\RE_{B}(E)$ is either a sum of suspensions of Morava $K$-theories or it is contractible. Moreover in the first case, 
$\id_{E}\wedge \eta\wkeq \ast$.
\end{prop}
\begin{proof} 

Since $\RE_{B}(E)\wedge X \wkeq \RE_{B}(E\wedge c^*X)$ and $E$ is a field, $\RE_{B}(E)\wedge X$ is free over $\RE_{B}(E)$ for any $X$ in $\SH$. Thus if $\RE_{B}(E)$ is not contractible, it is a sum of suspensions of Morava $K$-theories by \cite[Proposition 1.9]{hs:nil2}.

Assume the $E$ is not contractible. Then the $E$-algebra $\mathcal{S}_B\RE_B(E)$ is also not contractible. Since it is also free,   $P\wedge E$ is a retract of $\mathcal{S}_B\RE_{B}(E)$ for some invertible motivic spectrum $P$. 

Now, via the equivalence of \aref{lem:Seqv}, we have a commutative diagram of maps
$$
\xymatrix{
\Sing_{B}\RE_{B}(E) \ar[r]^-{\id\wedge \eta}\ar[rd]_-{\Sing(\id\wedge \eta^{top})} & 
\Sing_{B}\RE_{B}(E)\wedge S^{-\alpha} \ar[d]^{\wkeq} \\
& \mathcal{S}_B(\RE_{B}(E)\wedge S^{-1}).
}
$$

Since $\eta^{top}$ acts by zero on Morava $K$-theory, it follows that 
$\id_{\mathcal{S}_B\RE_{B}(E)}\wedge \eta\simeq\ast$ and hence 
$\id_{P\wedge E}\wedge \eta\simeq\ast$ and since $P$ is invertible, 
$\id_{E}\wedge \eta\simeq\ast$.
\end{proof}

\begin{remark}
If $F$ is real closed, then $X_F = *$ and \cite{bman} implies that $\SHA(F)[[-1]^{-1}]\simeq \SH$, the classical Spanier-Whitehead category.  Let $\EM(\ul W[\eta^{\pm 1}]\otimes \QQ)$ denote the Eilenberg-MacLane object associated with the homotopy module given by the rationalization of the Witt sheaf, and note that this object lives in $\SHA(F)[[-1]^{-1}]$ since $(2+[-1]\eta)\eta=0$.  Under Bachmann's equivalence, this object is sent to the rational Eilenberg-MacLane spectrum $H\mathbb{Q}$, which is a field in $\SH$.  Since every $\EM(\ul W[\eta^{\pm 1}]\otimes \QQ)$-module is $[-1]$-periodic, we may conclude that $\EM(\ul W[\eta^{\pm 1}]\otimes \QQ)$ is a motivic field spectrum.  Similarly, $\EM(\ul W[\eta^{\pm 1}]/p)$ is a motivic field spectrum for all odd primes $p$.  Of course, the complex Betti realizations of such spectra are trivial, but Bachmann's functor is a real Betti realization, \emph{i.e.}, the composition of $C_2$-equivariant Betti realization with geometric fixed points.
\end{remark}
  
\section{Questions}\label{sec:q}

This section contains a list of questions and commentary thereon.  It is the authors' hope that interested readers will take up these problems and advance our collective understanding of the tensor triangular geometry of stable motivic homotopy theory. In order to speak simultaneously about both compact cellular and all compact spectra, as before, we will write $\KK(F)$ to refer to either of $\Cc(F)$ and $\SHA(F)^c$.

\begin{question}\label{q:rigid}
Is  $\Spc(\KK(F))$ rigid in the sense that 
$\Spc(i^*)$ is a homeomorphism
whenever $i:F\hookrightarrow F'$ is an extension of algebraically closed or real closed fields? 
\end{question}

This question is motivated by the observations that the ring $GW(F)$ and the space $\Spec^h(K_{*}^{MW}(F))$ are rigid amongst these types of extensions.  Additionally, the  $p$-complete cellular categories are rigid for algebraically closed fields. 
Write $\Cc(F)^{\wedge}_{p}\subseteq \SHA(F)$ for  the thick subcategory generated by  $(S^{m\alpha})^{\wedge}_{p}$ for $m\in \Z$. 
Here $E^{\wedge}_{p} := L_{S^0/p}(E)$ is the Bousfield localization with respect to the mod-$p$ Moore spectrum. 

\begin{prop}
 Let $i:F\hookrightarrow F'$ be an extension of algebraically closed fields and $p$ a prime different from ${\rm char}(F)$.
Then $\Spc(\Cc(F)^{\wedge}_{p})\to \Spc(\Cc(F')^{\wedge}_{p})$ 
is a homeomorphism.
 \end{prop}
\begin{proof}
 By \cite[Theorem 1.1]{RO:rigidity},  
 $i^*:\SHA(F)^{\wedge}_p\to \SHA(F')^{\wedge}_p$ is full and faithful. On cellular categories it is also essentially surjective and so 
 $\Cc(F)^{\wedge}_p\to \Cc(F')^{\wedge}_p$ is an equivalence of categories.
\end{proof}

 Write $\Cc(F)_{\QQ}\subseteq \SHA(F)$ for the thick subcategory generated by $(S^{m\alpha})_{\Q}$, $m\in \Z$ and $\SHA(F)^{c}_{\QQ}$ for the thick subcategory generated by $(X_+)_{\Q}$ where $X$ is a smooth $F$-scheme.  Work of Cisinski-Deglise \cite[Corollary 16.2.14]{cd} implies that when $F$ is nonreal, 
 $\SHA(F)^c_\QQ$ is equivalent to the triangulated category of rational geometric motives $\mathrm{DM}_{gm}(F)_\QQ$ and similarly that
 $\Cc(F)_\QQ$ is equivalent to the triangulated category of rational mixed Tate motives, $\mathrm{DMT}(F)_\QQ$.   Write $\KK(F)_\QQ$ for either of $\Cc(F)_\QQ$ or $\SHA(F)^c_\QQ$.
  
\begin{question}\label{q:motives}
Is $\Spc(\KK(F)_\QQ)=\ast$ when $F$ is algebraically closed? 
\end{question}
\begin{remark}
Peter's study \cite{peter:dmt} of the tensor triangular geometry of $\mathrm{DMT}(F)_\QQ$ for $F$ an algebraic extension of $\QQ$ implies that $\Spc(\Cc(\overline{\QQ})_\QQ) = \ast$, so if \aref{q:rigid} has an affirmative answer, then $\Spc(\Cc(F)_\QQ)=\ast$ for all algebraically closed fields of characteristic $0$.  Also note that, over a finite field, assuming the Beilinson-Parshin conjecture and that rational and numerical equivalence on algebraic cycles agree, Kelly \cite{kelly} has shown that $\Spc(\SHA(\mathbb{F}_q)_{\Q}^c) = \ast$.
\end{remark}
 
\begin{remark}\label{rmk:cons}
\aref{q:motives} is hard: as remarked to us by Shane Kelly, it implies the Conservativity Conjecture.  This conjecture states that Betti realization 
${\rm DM}_{gm}(\CC)_{\Q} \to {\rm D}(\Q)$, on the category of rationalized geometric motives,
is a conservative functor, 
see \cite[Conjecture 2.1]{ayoub:cons}; equivalently $\RE:\SHA(\CC)^{c}_{\QQ}\to \SH^{\fin}_\QQ$ is conservative. Indeed, if $\Spc(\SHA(\CC)^{c}_{\QQ}) = \ast$, then by the classification of thick ideals \cite[Theorem 4.10]{Balmer:ttspectrum} there is a unique proper thick ideal in $\SHA(\CC)^{c}_{\QQ}$. In particular the proper thick ideals $\{X\in \SHA(\CC)^c_{\Q} \mid \RE(X)\wkeq \ast\}$ and $\{0\}$ are equal. 

In the other direction, we have the following proposition. 
Recall that a topological space is {\it local} if it has a unique  closed point.
\end{remark}

\begin{prop}\label{prop:cons}
 If the Conservativity Conjecture holds, then  
 $\Spc(\KK(\CC)_\QQ)$ is local.
\end{prop}
\begin{proof}
 The tensor triangulated category $\KK(\CC)_{\QQ}$ is rigid in the sense that
 every object is dualizable. By \cite[Proposition 4.2]{balmer:sss} these spaces are local if and only if $X\wedge Y \wkeq \ast$ implies $X\wkeq \ast$ or $Y\wkeq \ast$. If $X\wedge Y\wkeq \ast$, then $\RE(X\wedge Y) \wkeq \RE(X) \wedge \RE(Y) \wkeq \ast$. But since these are rational spectra, it follows that either $\RE(X)\wkeq \ast$ or $\RE(Y)\wkeq\ast$.
 Since $\SHA(\CC)^{c}_\QQ$ is equivalent to ${\rm DM}_{gm}(\CC)_{\Q}$ (see \cite[Corollary 16.2.14]{cd}), the Conservativity Conjecture implies that either $X\wkeq \ast$ or $Y\wkeq \ast$.   
\end{proof}

\begin{remark}
More generally, Balmer has shown in \cite[Theorem 1.2]{balmer:surj} that if $f:\KK\to\mathscr{L}$ is a tensor triangulated functor between essentially small, tensor triangulated categories, where $\KK$ is rigid, then $f$ is conservative if and only if the induced map $\Spc(f):\Spc(\mathscr{L})\to\Spc(\KK)$ is surjective on closed points.  This is consistent with \aref{rmk:cons} and \aref{prop:cons} because $\Spc(\SH_\QQ^\fin) = *$.
\end{remark}

\begin{remark}
When $F$ is perfect of finite $2$-\'etale cohomological dimension, Bachmann has shown in \cite[Theorem 15]{b:pic} that the functor $M:\SHA(F)^c\to {\rm DM}_{gm}(F)$ is conservative.  By \cite[Theorem 1.2]{balmer:surj}, we see that this is equivalent to $\Spc(M)$ being surjective on closed points.  Note, though, that $\Spc(M)$ is not surjective as its image does not intersect $(\rho^\bullet)^{-1}([F^\times],2)$.  Further study of $\Spc({\rm DM}_{gm}(F))$ and its relationship with $\Spc(\SHA(F)^c)$ would be quite interesting.
\end{remark}

The following question speculates on the fashion in which $\Spc(\KK(F))$ might be reconstructed from the tensor triangular spectra of the stable motivic homotopy categories of the real and algebraic closures of $F$.

\begin{question}\label{q:surj}
Let $X_F$ denote the set of orderings on $F$ and let $X_F^* = X_F\amalg\{\infty\}$.  For $\alpha\in X_F$, let $F_\alpha$ denote the real closure of $F$ with respect to $\alpha$; let $F_\infty = \overline{F}$ denote the algebraic closure of $F$.  For $\alpha\in X_F^*$, let $i_\alpha:F\hookrightarrow F_\alpha$ be a chosen embedding.  Each of these induces a map $\Spc(i_\alpha^*):\Spc(\KK(F_\alpha))\to \Spc(\KK(F))$.  Is the sum of these maps
\[
  \coprod_{\alpha\in X_F^*}\Spc(\KK(F_\alpha))\longrightarrow \Spc(\KK(F))
\]
surjective?  Can we explicitly describe $\Spc(\KK(F))$ as a quotient of $\coprod_{\alpha\in X_F^*}\Spc(\KK(F_\alpha))$?
\end{question}

Note that the analogue of the surjectivity question on $\Spec^h(\kmw*F)$ is true by Thornton's theorem.  If we could produce the quotient description and answer \aref{q:rigid} in the affirmative, this would give a robust description of $\Spc(\KK(F))$.

\begin{question}\label{q:CR}
What is the structure of $\Spc(\KK(\CC))$? of $\Spc(\KK(\RR))$?
\end{question} 

While simply stated, the authors believe that this may be the hardest of the questions.  (See \aref{rmk:suppCeta}.)  Were one to successfully tackle this problem and affirmatively answer \aref{q:rigid} and \aref{q:surj}, one would completely describe $\Spc(\KK(F))$, at least for characteristic $0$ fields.
 
\bibliographystyle{plain}
\bibliography{SpcSHA1}

\end{document}